\documentclass[leqno]{amsart}
\usepackage{newclude}
\usepackage{amsfonts,amssymb,amsmath,amsgen,amsthm}
\allowdisplaybreaks
\usepackage{pdfsync}
\usepackage{ucs}
\usepackage{enumerate}

\usepackage{relsize} 
\usepackage{bigints}

\theoremstyle{plain}
\newtheorem{theorem}{Theorem}[section]
\newtheorem{definition}[theorem]{Definition}

\newtheorem{lemma}[theorem]{Lemma}
\newtheorem{corollary}[theorem]{Corollary}

\theoremstyle{remark}

\def\R{{\mathbf R}}
\def\T{{\mathbf T}}
\def\N{{\mathbf N}}
\def\Z{{\mathbf Z}}

\def\Q{\mathbf Q}
\def\dd{\mathrm d}

\def\({\left(}
\def\){\right)}
\def\<{\left\langle}
\def\>{\right\rangle}

\def\1{{\mathbf 1}}

\newcommand{\D}{\mathbf{D}}

\newcommand{\Tv}{\mathbf{T}_{\nu}}
\newcommand{\Sv}{\mathbf{S}_{\nu}}
\newcommand{\Av}{\mathbf{A}_{\nu}}
\newcommand{\Tvn}{\mathbf{T}_{\nu,n}}
\newcommand{\Svn}{\mathbf{S}_{\nu,n}}
\newcommand{\Sk}{\mathbf{S}_{\kappa}}
\newcommand{\Skn}{\mathbf{S}_{\kappa,n}}

\newcommand{\tun}{\tilde{u}_n}
\newcommand{\Qn}{\mathcal{Q}_n}

\newcommand{\bdl}{\beta_{\delta}^l}
\newcommand{\nyb}{\nabla_y\bdl}

\DeclareMathOperator{\diver}{div}
\DeclareMathOperator{\dive}{div}

\numberwithin{equation}{section}

\date\today

\title[Global existence of FEWS to QNS with non-trivial far-field]{Global existence of finite energy weak solutions to the Quantum Navier-Stokes equations with non-trivial far-field behavior}

\author[P. Antonelli]{Paolo Antonelli}
\address{
Gran Sasso Science Institute, viale Francesco Crispi, 7, 67100 L'Aquila, Italy}
\email{paolo.antonelli@gssi.it}

\author[L.E. Hientzsch]{Lars Eric Hientzsch}
\address{
Univ. Grenoble Alpes, CNRS, Institut Fourier, 100 rue des Math{\'e}matiques, 38610 Gi{\`e}res, France}
\email{larseric.hientzsch@univ-grenoble-alpes.fr}

\author[S. Spirito]{Stefano Spirito}
\address{
DISIM - Dipartimento di Ingegneria e Scienze dell'€™Informazione e Matematica, Universit{\`a} degli Studi dell'€™Aquila, Via Vetoio, 67100 L'€™Aquila, Italy}
\email{stefano.spirito@univaq.it}


\begin{document}
\begin{abstract}
We prove global existence of finite energy weak solutions to the quantum Navier-Stokes equations in the whole space with non trivial far-field condition in dimensions d = 2,3. The vacuum regions are included in the weak formulation of the equations. Our method consists in an invading domains approach. More precisely, by using a suitable truncation argument we construct a sequence of approximate solutions. The energy and the BD entropy bounds allow for the passage to the limit in the truncated formulation leading to a finite energy weak solution. Moreover, the result is also valid in the case of compressible Navier-Stokes equations with degenerate viscosity. 
\end{abstract}

\maketitle

\section{Introduction}

The aim of the paper is to study the Cauchy Problem associated to the following Quantum-Navier-Stokes equations (QNS) in $(0,T)\times \R^d$ for $d=2,3$, 
\begin{equation}\label{eq:QNSFEWS}
\begin{aligned}
\begin{cases}
&\partial_t\rho+\dive (\rho u)=0,\\
&\partial_t(\rho u)+\dive\left(\rho u \otimes u \right)+\nabla p(\rho)=2\nu \dive(\rho\D u)+2\kappa^2\rho \nabla\left(\frac{\Delta\sqrt{\rho}}{\sqrt{\rho}}\right),
\end{cases}
\end{aligned}
\end{equation}
with the far-field behaviour  
\begin{equation}\label{eq:farfield}
\rho \rightarrow 1 \hspace{1cm} \text{as} \hspace{0.5cm} |x|\rightarrow \infty.
\end{equation}
The unknowns are the mass density $\rho$ and the velocity field $u$, the pressure is given by the $\gamma$-law with $\gamma>1$ and $\nu$ and $\kappa$ are the viscosity and the capillarity coefficients, respectively.
The energy of the system \eqref{eq:QNSFEWS} is defined as
\begin{equation}\label{eq:energy}
E(t)=\int_{\R^d}\frac{1}{2}\rho|u|^2+2\kappa^2|\nabla\sqrt{\rho}|^2+F(\rho) \dd x,
\end{equation}
where the internal energy reads
\begin{equation}\label{eq:ren_ieQNSFEWS}
    F(\rho)=\frac{\rho^{\gamma}-1-\gamma(\rho-1)}{\gamma(\gamma-1)}.
\end{equation}
Notice that the integrability of \eqref{eq:ren_ieQNSFEWS} encodes the far-field condition for finite energy solutions.
The Quantum Navier-Stokes equations falls within the more general class of Navier-Stokes-Korteweg systems which include capillarity effects in the dynamics of the fluid \cite{K}. Such systems are given by the following equations
\begin{equation}\label{eq:Kortewegintro}
\begin{aligned}
&\partial_t\rho+\diver(\rho u)=0,\\
&\partial_t(\rho u)+\diver(\rho u \otimes u)+\nabla P(\rho)=2\nu\diver(\mathbb{S})+ \kappa^2 \diver(\mathbb{K}),
\end{aligned}
\end{equation}
where the viscous stress tensor $\mathbb{S}=\mathbb{S}(\nabla u)$ is determined by
\begin{equation*}
\mathbb{S}=h(\rho)\D u+g(\rho)\diver u \mathbf{I},
\end{equation*}
and $h$ and $g$ are called viscosity coefficients satisfying $h(\rho)\geq 0$ and $h(\rho)+d\,g(\rho)\geq 0$. 
The capillary term $\mathbb{K}=\mathbb{K}(\rho,\nabla\rho)$ reads
\begin{equation*}
\mathbb{K}=\left(\rho\diver(k(\rho)\nabla\rho)-\frac{1}{2}(\rho k'(\rho)-k(\rho)|\nabla\rho|^2)\right)\mathbb{I}-k(\rho)\nabla\rho\otimes\nabla\rho,
\end{equation*}
where $k(\rho)\geq 0$ is called the capillarity coefficient. 
The stress tensor $\mathbb{K}$ describes the capillarity effects and it was originally introduced by Korteweg \cite{K} and rigorously derived in the present form by \cite{DS}. The description of quantum fluids also relies on similar systems, e.g. the QHD system, namely the inviscid counter-part ($\nu=0$) of \eqref{eq:QNSFEWS}, arises as model for superfluidity and Bose-Einstein condensation \cite{LL}. The Cauchy Problem with trivial far-field condition is investigated in \cite{AM, AM16} and the QHD system with non-trivial far-field condition will be addressed in the ongoing work \cite{AHM}.
The quantum Navier--Stokes equation that we study below is obtained from \eqref{eq:Kortewegintro} by setting $h(\rho)=\rho$, $g(\rho)=0$ and $k(\rho)=\frac{1}{\rho}$. \par
Alternatively, the QNS system can also be derived as the moment closure of a Wigner equation with a BGK type collision term \cite{BM, JMi}, see also \cite{J} where several dissipative quantum fluid models are derived by means of a moment closure of (quantum) kinetic equations with appropriate choices of the collision term. In the absence of capillary effects, namely $\kappa=0$, system \eqref{eq:QNSFEWS} reduces to the compressible Navier-Stokes equations with degenerate viscosity. The global existence of weak solutions in two and three dimensions has been obtained in \cite{VY16} and \cite{LX} with vanishing boundary conditions at infinity. Concerning the quantum Navier--Stokes system, the existence of finite energy weak solutions on the torus $\T^d$, for $d=2,3$ has been proved in \cite{AS} and \cite{LV16}, see also \cite{AS15, LZZ}.
Global existence of weak solutions to the isothermal quantum Navier-Stokes equations, i.e. $\gamma=1$, has been proved in \cite{CCH}. The result of \cite{CCH} is proven on $\R^d$ for $d=2,3$ with vanishing boundary conditions at infinity by means of an invading domain approach. Moreover, for some different choices of viscosity and capillarity coefficients the global existence has been proved in \cite{AS19, AS18, BVY19}.
 In general, the analysis of fluids with density dependent viscosity, even without capillary tensor, requires new tools compared to the case of constant viscosity, see Lions-Feireisl theory \cite{L96, FN01} and also \cite{BJ}. This is due to a loss of
control on uniform bounds on the velocity field in vacuum regions. On the other hand, in the case of degenerate viscosities, the Bresch-Desjardins entropy estimate \cite{BD} gives further regularity properties on the mass density, see \eqref{ineq:BDEQNS}. 
When considered with capillary effects, the nonlinear structure of the dispersive tensor depending on the density and its derivatives entails additional mathematical difficulties, see for instance \cite{BD, BDL, BNV16}.

The choice of considering system \eqref{eq:QNSFEWS} with non-trivial conditions at infinity as in \eqref{eq:farfield} is motivated by its applications to the study of singular limits such as the low Mach number limit \cite{AHM18, AHMHyp}.
Moreover, this is also the right framework where it is possible to consider a wide class of special solutions to \eqref{eq:QNSFEWS}, such as traveling waves, solitons, viscous shocks, etc.

We remark that our result of global existence of finite energy weak solutions also applies to the compressible Navier-Stokes equations with degenerate viscosity, namely for $\kappa=0$. To the best of our knowledge, the Cauchy Problem for \eqref{eq:QNSFEWS} with $\kappa\geq 0$ and far-field behavior \eqref{eq:farfield} has not been previously investigated in literature in the class of weak solutions for $d=2,3$. Local strong solutions have been constructed in \cite{GS18, LPZ16,LPZ} for \eqref{eq:QNSFEWS} with $\nu>0$ and $\kappa=0$. For $d=1$, $\nu>0$ and $\kappa=0$, existence and uniqueness of global strong solutions with \eqref{eq:farfield} has been shown in \cite{MV08, H18}.

Let us now comment on the definition of finite energy weak solutions of \eqref{eq:QNSFEWS}. Without further regularity assumptions neither $u$, $\nabla u$ nor $\frac{1}{\sqrt\rho}$ are defined almost everywhere due to the possible presence of vacuum, this is somehow reminiscent to the analysis of the QHD system \cite{AM, AM16}. On the other hand, the quantity $\Lambda:=\sqrt{\rho}\,u$ is well-defined almost everywhere and then the definition of   finite energy weak solutions is given in terms $(\sqrt{\rho}, \Lambda)$. The same difficulty arises in the context of barotropic Navier-Stokes equations with density dependent viscosity, see for instance \cite{LX}.\par

 Moreover, in general it is not possible to infer the following energy inequality
\begin{equation}\label{eq:EIformal}
    E(t)+2\nu \int_0^t\int_{\R^d}\rho\left|\D u\right|^2\dd x\dd t\leq E(0).
\end{equation}
which is replaced by a weaker version. To this end, we define the tensor $\Tv\in L^2(\R\times\R^d)$ satisfying  
\begin{equation}\label{eq:visctensor}
    \sqrt{\nu}\sqrt{\rho}\Tv=\nu\nabla(\rho u)-2\nu \Lambda \otimes\nabla\sqrt{\rho} \qquad \text{in} \quad \mathcal{D}'((0,T)\times \R^d)).
\end{equation}
By denoting $\Sv=\Tv^{sym}$, we see that for smooth solutions we have $\sqrt{\nu}\sqrt{\rho}\Sv=\nu\rho\D u$ and then \eqref{eq:EIformal} reads as \eqref{eq:energyQNS}.\par
Analogously, for the capillary tensor we use the identities 
\begin{equation*}
\frac{1}{2}\rho\nabla\left(\frac{\Delta\sqrt{\rho}}{\sqrt{\rho}}\right)=\frac{1}{4}\nabla\Delta\rho- \diver(\nabla\sqrt{\rho}\otimes\nabla\sqrt{\rho})=\dive\left(\sqrt{\rho}(\nabla^2\sqrt{\rho}-4\nabla\rho^{\frac{1}{4}}\otimes\nabla\rho^{\frac{1}{4}})\right)
\end{equation*}
%
%
%
which turn out to be  well-defined in view of the regularity inferred by the energy and the Bresch-Desjardins entropy estimates. 
 For the sake of consistency with the literature regarding (quantum) Navier-Stokes equations, we do not use the hydrodynamic variable $\Lambda$ in this paper. However, we stress that whenever the symbol $\sqrt{\rho}u$ appears it should be read as $\Lambda$. \par
Our method consists in an invading-domains approach, as in \cite{CCH}, see also \cite{L} and \cite{FN}, relying on the existence of weak solutions on flat torus \cite{LV16}, \cite{AS}, \cite{LZZ} which in addition satisfy a truncated formulation of the momentum equations. More precisely, given initial data of finite energy we construct periodic initial data on a sequence of invading tori. On each of them, \cite{LV16} provides a periodic truncated weak solution. Next, we show that these solutions provide a sequence of approximate truncated weak solutions on the whole space. In the limit, we recover a truncated weak solution to \eqref{eq:QNSFEWS} on the whole space which we show is in particular a finite energy weak solution. 
Contrarily to the aforementioned literature, here we need to overcome the further difficulty given by the non-trivial conditions at infinity for the mass density; this will be dealt with by the  appropriate choice for the internal energy (as already exposed above) and of the spatial cut-offs, see Subsection \ref{sec:periodicdata} below. \par

This paper is structured as follows.
In Section \ref{sec:preQNS}, we give the precise Definition of finite energy weak solution and state our main results. Section \ref{sec:periodic} is dedicated to the construction of suitable periodic initial data on growing tori $\T_n^d$ for which we postulate the existence of finite energy weak solutions based on \cite{LV16}. In particular, the periodic weak solutions satisfy a truncated formulation of the momentum equation that will be crucial for the sequel.
In Section \ref{sec:ext}, the sequence of periodic solutions on invading domains is extended to a sequence of approximate truncated weak solutions on the whole space with non-trivial far-field behavior \eqref{eq:farfield}. 
The limit of the sequence of approximate solutions provides a truncated finite energy weak solution to \eqref{eq:QNSFEWS}. Finally, we present the proof of the main Theorem in Section \ref{sec:weak}, namely we show that the obtained truncated finite energy weak solutions are finite energy weak solutions to \eqref{eq:QNSFEWS}.

\section{Definition and main results}\label{sec:preQNS}
As already mentioned, the strategy for proving the existence Theorem \ref{thm:mainQNS} goes through constructing suitable solutions on domains $\T_n^d=\R^d\slash n\Z^d$ with $n\in \N$. For the sake of generality we give the Definition of finite energy weak solutions for an arbitrary domain $\Omega$, which will be $\Omega=\R^d, \T^d$ or $\Omega=\T_n^d$ respectively according to our purposes.

\begin{definition}\label{defi:FEWS}
A pair $(\rho,u)$ with $\rho\geq 0$ is said to be a finite energy weak solution of the Cauchy Problem \eqref{eq:QNSFEWS} posed on $[0,T)\times \Omega$ complemented with initial data $(\rho^0,u^0)$ if 
\begin{enumerate}[(i)]
\item Integrability conditions
\begin{equation*}
\begin{aligned}
&\sqrt{\rho}\in L_{loc}^{2}((0,T)\times \Omega); \quad 
\sqrt{\rho}u\in L_{loc}^{2}((0,T)\times \Omega); \quad  \quad \nabla\sqrt{\rho}\in L_{loc}^{2}((0,T)\times \Omega); \\
&\nabla\rho^{\frac{\gamma}{2}}\in L_{loc}^{2}((0,T)\times \Omega); \quad \Tv\in L_{loc}^2((0,T)\times \Omega); \quad {\kappa}\nabla^2\sqrt{\rho}\in L_{loc}^{2}((0,T)\times \Omega);\\
&\sqrt{\kappa}\nabla\rho^{\frac{1}{4}}\in L_{loc}^4((0,T)\times\Omega);
\end{aligned}
\end{equation*}
\item Continuity equation
\begin{equation*}
\int_{\Omega}\rho_{0}\phi(0)+\int_0^T\int_{\Omega}\rho \phi_t+\sqrt{\rho}\sqrt{\rho}u\nabla\phi=0,
\end{equation*}
for any $\phi\in C_c^{\infty}([0,T)\times \Omega)$.
\item Momentum equation
\begin{equation*}
\begin{aligned}
&\int_{\Omega}\rho_{0}u_{0}\psi(0)+\int_0^T\int_{\Omega}\sqrt{\rho}\sqrt{\rho}u \psi_t+(\sqrt{\rho}u\otimes\sqrt{\rho}u)\nabla\psi+\rho^{\gamma}\dive \psi\\
&-2\nu \int_0^T\int_{\Omega}\left(\sqrt{\rho}u\otimes \nabla \sqrt{\rho}\right)\nabla \psi-2\nu \int_0^T\int_{\Omega}\left(\nabla\sqrt{\rho}\otimes \sqrt{\rho}u\right)\nabla \psi\\
&+\nu \int_0^T\int_{\Omega}\sqrt{\rho}\sqrt{\rho}u\Delta\psi+\nu \int_0^T\int_{\Omega}\sqrt{\rho}\sqrt{\rho}u\nabla\dive\psi\\
&-4\kappa^2 \int_0^T\int_{\Omega}\left(\nabla\sqrt{\rho}\otimes \nabla\sqrt{\rho}\right)\nabla\psi+2\kappa^2\int_0^T\int_{\Omega}\sqrt{\rho}\nabla\sqrt{\rho}\nabla\dive\psi=0,
\end{aligned}
\end{equation*}
for any $\psi\in C_c^{\infty}([0,T)\times \Omega;\Omega)$.
\item There exists a tensor $\Tv\in L^2((0,T)\times\Omega)$  satisfying identity \eqref{eq:visctensor}, namely
\begin{equation*}
    \sqrt{\nu}\sqrt{\rho}\Tv=\nu\nabla(\rho u)-2\nu\Lambda \otimes \nabla\sqrt{\rho} \qquad \text{in} \quad \mathcal{D}'((0,T)\times \R^d)),
\end{equation*}
such that the following energy inequality holds for a.e. $t\in[0,T]$,
\begin{equation}\label{eq:energyQNS}
E(t)+2 \int_0^t \int_{\Omega}|\Sv|^2\dd x\dd t\leq C E(0),
\end{equation}
where $\Sv$ is the symmetric part of $\Tv$ and $E$ as defined in \eqref{eq:energy}.
\item Let 
\begin{equation*}
B(t)=\int_{\Omega}\frac{1}{2}\left|\sqrt{\rho}u\right|^2+(2\kappa^2+4\nu^2)\left|\nabla\sqrt{\rho}\right|^2+F(\rho)\dd x.
\end{equation*}
Then for a.e. $t\in[0,T]$,
\begin{align}\label{ineq:BDEQNS}
\begin{split}
&B(t)
+\int_0^t\int_{\Omega}\frac{1}{2}\left|\Av\right|^2\dd x \dd s+\nu\kappa^2\int_0^t\int_{\Omega}\left|\nabla^2\sqrt{\rho}\right|^2+|\nabla\rho^{\frac{1}{4}}|^4\dd x\dd s+\nu\int_0^t\int_{\Omega}\left|\nabla\rho^{\frac{\gamma}{2}}\right|^2\dd x\dd s\\
&\leq C \int_{\Omega}\frac{1}{2}\left|\sqrt{\rho^0}u_{0}\right|^2+(2\kappa^2+4\nu)\left|\nabla\sqrt{\rho_{0}}\right|^2+F(\rho^0)\dd x,
\end{split}
\end{align}
where $\Av=\Tv^{asym}$, with $\Tv$ defined as in the previous point.
\end{enumerate}
\end{definition}
Notice that the far-field behavior is encoded in the definition of the energy functional \eqref{eq:energy}.
Our main result states the existence of global finite energy weak solutions to \eqref{eq:QNSFEWS} with non-vanishing density at infinity.

\begin{theorem}\label{thm:mainQNS}
Let $d=2,3$ and $\gamma>1$. Given initial data $(\rho^0,u^0)$ of finite energy, there exists a global finite energy weak solution to \eqref{eq:QNSFEWS} on $\R^d$ satisfying the far-field condition \eqref{eq:farfield}.
\end{theorem}

When $\kappa=0$, system \eqref{eq:QNSFEWS} reduces to the compressible Navier-Stokes equations with degenerate viscosity for which we have the following existence result.

\begin{corollary}\label{coro:degenerate}
Let $d=2,3$ and $\kappa=0$, $\gamma>1$. Given initial data $(\rho^0,u^0)$ of finite energy and BD entropy, there exists a global finite energy weak solution to \eqref{eq:QNSFEWS} on $\R^d$ satisfying the far field condition \eqref{eq:farfield}.
\end{corollary}

The statement of Theorem \ref{thm:mainQNS} remains true if we consider system \eqref{eq:QNSFEWS} on $\R^d$ with trivial far-field behavior, i.e. 
\begin{equation*}
    \rho\rightarrow 0, \quad \text{as} \quad |x|\rightarrow \infty.
\end{equation*}
In this case, the internal energy does not need to be renormalized, namely \eqref{eq:renenergyQNS} is substituted by
\begin{equation}\label{eq:energygammalaw}
    F(\rho)=\frac{1}{\gamma-1}\rho^{\gamma}.
\end{equation}
The following is the analogue result on the whole space to \cite{AS, LV16} for $\kappa>0$ and \cite{LX, VY16} for $\kappa=0$ respectively. 
\begin{theorem}\label{thm:mainQNS2}
Let $d=2,3$ and $\nu>0,\kappa\geq 0$ and $\gamma>1$. Given initial data $(\rho^0,u^0)$ of finite energy and BD entropy, there exists a global finite energy weak solution to \eqref{eq:QNSFEWS} on $\R^d$ with trivial far-field behavior.
\end{theorem}
Notice that \cite{LX} provides an existence result for the compressible Navier-Stokes with degenerate visocosity, i.e. $\kappa=0$ on the whole space with trivial far-field for a generalized viscosity tensor but with restrictions on the range of $\gamma$ for $d=3$. For $\gamma=1$ and $\kappa\geq 0$ existence has been proven in \cite{CCH}.

\subsection{Truncation functions}
We introduce the truncation functions by means of which we shall construct the aforementioned solutions to the truncated formulation of \eqref{eq:QNSFEWS}.

\begin{definition}\label{defi:truncation}
Let $\overline{\beta}:\R\rightarrow\R$ be an even positive compactly supported smooth function such that $\overline{\beta}(z)=1$ for $z\in[-1,1]$ and $supp(\overline{\beta})\subset(-2,2)$ and $0\leq\overline{\beta}\leq 1$. Further, we define $\tilde{\beta}:\R\rightarrow \R$ as 
\begin{equation*}
    \tilde{\beta}(z)=\int_0^z\overline{\beta}(s)\dd s.
\end{equation*}
For any $\delta>0$, we define $\overline{\beta}_{\delta}(z):=\overline{\beta}(\delta z)$ and $\Tilde{\beta}_{\delta}(z):=\frac{1}{\delta}\Tilde{\beta}(\delta z)$.
Given $y\in \R^3$, we denote 
\begin{equation*}
    \hat{\beta}_{\delta}(y):=\prod_{l=1}^3\overline{\beta}_{\delta}(y_l),
\end{equation*}
further
\begin{equation*}
    \beta_{\delta}^1(y)=\int_0^{y_1}\hat{\beta}_{\delta}(y_1',y_2,y_3)\dd y_1'=\Tilde{\beta}_{\delta}(y_1)\overline{\beta}_{\delta}(y_2)\overline{\beta}_{\delta}(y_3).
\end{equation*}
The functions $ \beta_{\delta}^2(y),  \beta_{\delta}^3(y)$ are defined analogously. 

\end{definition}
We summarize some properties of the truncation functions introduced in Definition \ref{defi:truncation}.
\begin{lemma}\label{lem:truncation}
Let $\delta>0$,  $\beta_\delta^l$ as in Definition \ref{defi:truncation} and $M:=\|\beta\|_{W^{2,\infty}}$. Then, there exists $C=C(M)>0$ such that the following bounds hold.
\begin{enumerate}
    \item for $1\leq l \leq d$,
    \begin{equation*}
        \|\beta_\delta^l\|_{L^{\infty}}\leq \frac{C}{\delta}, \qquad \|\nabla\beta_\delta^l \|_{L^{\infty}}\leq C, \qquad \|\nabla^2\beta_\delta^l \|_{L^{\infty}}\leq C \delta,
    \end{equation*}
    \item as $\delta$ goes to $0$, for every $y\in \R^d$ and any $1\leq l \leq d$,
    \begin{equation*}
        \beta_\delta^l(y)\rightarrow y_l, \qquad \nabla\beta_{\delta}^{l}(y)\rightarrow e_l,
    \end{equation*}
    where $e_l\in \R^d$ such that $e_l^i=1$ for $i=l$ and $e_l^i=0$ otherwise. 
\end{enumerate}
\end{lemma}

\section{Solutions on periodic domains}\label{sec:periodic}
\subsection{Construction of periodic initial data}\label{sec:periodicdata}
Given initial data of finite energy $(\sqrt{\rho^0},\sqrt{\rho^0}u^0)$ for the problem \eqref{eq:QNSFEWS} posed on $\R^d$ with far-field boundary conditions \eqref{eq:farfield}, we construct a sequence of initial data to the periodic problem on $\T_n^d$.
To that end, we consider a smooth cut-off function $\eta\in C_c^{\infty}(\R^d)$ with $supp(\eta)\subset [-1,1]^d$ such that 
\begin{equation}\label{eq:eta}
    \mathbf{1}_{[-\frac{1}{2},\frac{1}{2}]^d}\leq \eta\leq  \mathbf{1}_{[-1,1]^d}, \qquad \text{and} \quad \eta_n(x)=\eta\left(\frac{x}{n}\right).
\end{equation}
We denote $\Qn=[-n,n]^d\backslash(-\frac{n}{2},\frac{n}{2})$ and notice that $supp(\nabla\eta_n)\subset \Qn$ as well as $\|\nabla\eta_n\|_{L^{\infty}(\R^d)}=O(\frac{1}{n})$. 
We define the sequence of initial data $(\sqrt{\rho_n^0},\sqrt{\rho_n^0}u_n^0)$ on $\T_n^d$ as follows
\begin{equation}\label{eq:initialdata}
    \sqrt{\rho_n^0}=\sqrt{\rho^0}\eta_n+(1-\eta_n), \qquad \sqrt{\rho_n^0}u_n^0=\sqrt{\rho^0}u^0\mathbf{1}_{(-n,n)^d}.
\end{equation}

We observe that due to the choice of the cut-off functions, the sequence of initial data $(\sqrt{\rho_n^0}, \Lambda_n^0)$ can be extended to periodic functions and may therefore be considered as functions defined on $\T_n^d$. A similar construction has recently been used in \cite{CCH} for the existence of weak solutions to isothermal fluids. In the present setting, we additionally need to take into account the non-vanishing conditions at infinity leading to a lack of integrability.  We collect some of its properties.
For that purpose, we recall the uniform estimates satisfied by the initial data $(\sqrt{\rho^0},\Lambda^0)$ on $\R^d$, see also \cite{AHM18} for more details. For a generic pair $(\sqrt{\rho},\sqrt{\rho}u)$ of finite energy we have the following bounds.

\begin{lemma}\label{lem:dataQNS}
Let $(\sqrt{\rho},\sqrt{\rho}u)$ be such that $E(\sqrt{\rho},\sqrt{\rho}u)<+\infty$. Then 
\begin{equation*}
    \sqrt{\rho}-1\in H^1(\R^d), \quad \rho-1\in L_{2}^{\gamma}(\R^d)\cap L^2(\R^d), \quad \sqrt{\rho}u\in L^2(\R^d).
\end{equation*}
\end{lemma}

\begin{proof}
The bound $E(\sqrt{\rho},\sqrt{\rho}u)<+\infty$ implies $\nabla\sqrt{\rho}\in L^2(\R^d)$. By a convexity argument, see Appendix A in \cite{AHM18} or also Appendix E in \cite{LM98}, the bound $F(\rho)\in L^1(\R^d)$ yields the estimate in the Orlicz space  $\rho-1\in L_2^{\gamma}(\R^d)$, i.e. there exists $C>0$ such that
\begin{equation*}
    \int_{\R^d}|\rho-1|^2\mathbf{1}_{\{|\rho-1|\leq\frac{1}{2}\}}+|\rho-1|^\gamma\mathbf{1}_{\{|\rho-1|>\frac{1}{2}\}}\dd x\leq C.
\end{equation*}
If $\gamma\geq 2$, this implies $\rho-1\in L^2(\R^d)$. The inequality $|\sqrt{\rho}-1|\leq |\rho-1|$ yields $\sqrt{\rho}-1\in L_2^{\gamma}(\R^d)$. If $\gamma<2$, we notice that the set $\{|\rho-1|>\frac{1}{2}\}$ is of finite Lebesgue measure and $\nabla\sqrt{\rho}\in L^2(\R^d)$ implies $\sqrt{\rho}-1\in L_{loc}^2(\R^d)$. Hence, there exists $C_1>0$ such that
\begin{equation*}
     \int_{\R^d}|\sqrt{\rho}-1|^2\mathbf{1}_{\{|{\rho}-1|>\frac{1}{2}\}}\dd x\leq C_1,
\end{equation*}
therefore $\sqrt{\rho}-1\in H^1(\R^d)$. For real $s\geq 0$ such that $|s^2-1|>\frac{1}{2}$, there exists $c>0$ such that $(s^2-1)^2\leq c (s-1)^4$ implying that there exists $C_2>0$, such that
\begin{equation*}
    \int_{\R^d}|\rho-1|^2\mathbf{1}_{\{|\rho-1|>\frac{1}{2}\}}\dd x\leq \int_{\R^d}|\sqrt{\rho}-1|^4\mathbf{1}_{\{|{\rho}-1|>\frac{1}{2}\}}\dd x \leq C_2.
\end{equation*}
Hence, we conclude $\rho-1\in L^2(\R^d)$. The remaining bounds are immediate consequences of the finite energy assumption.
\end{proof}

Next, we infer uniform bounds (in $n$) for the sequence of periodic initial data defined in \eqref{eq:initialdata}. 

\begin{lemma}\label{lem:data}
Given $(\sqrt{\rho^0},\sqrt{\rho^0}u^0)$ such that $E_{\R^d}(\sqrt{\rho^0},\sqrt{\rho^0}u^0)<\infty$, the sequence of initial data $(\sqrt{\rho_n^0},\sqrt{\rho_n^0}u_n^0)$ defined by \eqref{eq:initialdata} satisfies the following. There exists an absolute constant $C>0$ such that denoting $\rho_n^0=(\sqrt{\rho_{n}^0})^2$, one has
\begin{equation}\label{eq:energyperiodicdata}
\int_{\T_n^d}\frac12 \left|\nabla\sqrt{\rho_n^0}\right|^2+\frac12\left|\sqrt{\rho_n^0}u_n^0\right|^2+\frac{1}{\gamma-1}\rho_n^{0,\gamma}\dd x\leq C n^d,
\end{equation}
Further,
\begin{align}
\lim_{n\rightarrow \infty}\int_{\T_n^d}\frac{1}{2} \left|\sqrt{\rho_n^0}-1\right|^2&= \int_{\R^d}\frac{1}{2}|\sqrt{\rho^0}-1|^2\dd x, \label{eq:datal2}\\
\lim_{n\rightarrow \infty}\int_{\T_n^d}\frac{1}{2} \left|\nabla\sqrt{\rho_n^0}\right|^2&= \int_{\R^d}\frac{1}{2}|\nabla\sqrt{\rho^0}|^2\dd x,\label{eq:grad}\\
\lim_{n\rightarrow \infty}\int_{\T_n^d} F(\rho_{n}^0)&=\int_{\R^d}F(\rho^0)\dd x, \label{eq:renenergyQNS}\\
\lim_{n\rightarrow \infty}\int_{\T_n^d}\frac12 \rho_n^0\left|u_n^0\right|^2&= \int_{\R^d}\frac{1}{2}\rho^0|u^0|^2\dd x,\label{eq:Lambda}
    \end{align}

In particular, there exists $C>0$ such that $E_{T_n^d}(\sqrt{\rho_n^0},\sqrt{\rho_n^0}u_n^0)\leq C$ uniformly in $n$, with $E$ as defined in \eqref{eq:energy} and
\begin{equation*}
    \lim_{n\rightarrow\infty}E_{T_n^d}\left(\sqrt{\rho_n^0},\sqrt{\rho_n^0}u_n^0\right)=E(\sqrt{\rho^0},\sqrt{\rho^0}u^0).
\end{equation*}
\end{lemma}

We emphasize that the bound on the right-hand side of \eqref{eq:energyperiodicdata} depends on $n$ through the internal energy in view of the far-field condition \eqref{eq:farfield}. 

\begin{proof}
Given $(\sqrt{\rho^0},\sqrt{\rho^0} u^0)$ of finite energy let the sequence $(\sqrt{\rho_n^0},\sqrt{\rho_n^0}u_n^0)$ be defined by \eqref{eq:initialdata}. Inequality \eqref{eq:energyperiodicdata} will then follow from the convergences \eqref{eq:grad},  \eqref{eq:renenergyQNS} and \eqref{eq:Lambda}. Since \eqref{eq:datal2} and \eqref{eq:grad} imply that  $\|\sqrt{\rho_n^0}-1\|_{H^1(\T_n^d)}\leq C$ uniformly in $n$, we conclude that $\sqrt{\rho_n^0}\in H^1(\T_n^d)$ for all $n\in \N$. From \eqref{eq:renenergyQNS}, one has that $\rho_n^0-1\in L_2^{\gamma}(\T_n^d)$, thus also $(\rho_n^0)^{\gamma}\in L^1(\T_n^d)$ by checking that 
\begin{equation*}
    \int_{\T_n^d}(\rho_n^0)^{\gamma}\mathbf{1}_{\{|\rho_n^0-1|\leq \frac{1}{2}\}}\dd x\leq \int_{\T_n^d}1+L_{\gamma}|\rho-1|\mathbf{1}_{\{|\rho_n^0-1|\leq \frac{1}{2}\}} \dd x\leq C_{n,d},
\end{equation*}
where $C_{n,d}$ is proportional to the volume of $\T_n^d$.
Thus, it only remains to prove \eqref{eq:datal2} - \eqref{eq:Lambda}. Let us show \eqref{eq:datal2}. 
We notice that 
\begin{equation*}
    \sqrt{\rho_n^0}-1=(\sqrt{\rho^0}-1)\eta_n,
\end{equation*}
and thus 
\begin{equation*}
    | \sqrt{\rho_n^0}-1|\leq |\sqrt{\rho^0}-1|.
\end{equation*}
We conclude by means of the dominated convergence Theorem that 
\begin{equation*}
    \lim_{n\rightarrow\infty}\int_{\T_n^d}\left|\sqrt{\rho_n^0}-1\right|^2\dd x
   =\int_{\R^d}|\sqrt{\rho^0}-1|^2\dd x
\end{equation*}
Similarly in order to derive \eqref{eq:grad} , we recall that $supp(\nabla\eta_n)\subset \Qn$ and observe that $\Q_n$ has measure of order $O(n^d)$. We have that for all $n\in \N$,
\begin{equation*}
    \begin{aligned}
    &\int_{\T_n^d}\frac{1}{2}|\nabla\sqrt{\rho_{0}^n}|^2\dd x= \int_{\R^d}\frac{1}{2}\left|(\nabla\sqrt{\rho^0})\eta_n+(\sqrt{\rho^0}-1)\nabla\eta_n\right|^2\dd x\\
    &=\frac{1}{2}\int_{\R^d}\eta_n^2\left|\nabla\sqrt{\rho^0}\right|^2+2\eta_n(\sqrt{\rho^0}-1)\nabla\sqrt{\rho^0}\cdot\nabla\eta_n+(\sqrt{\rho^0}-1)^2|\nabla\eta_n|^2\dd x\\
    &\leq \int_{\R^d}\frac{1}{2}|\nabla\sqrt{\rho^0}|^2\dd x
    +\left(\|\nabla\sqrt{\rho^0}\|_{L^2(\Qn)}\|\sqrt{\rho}_0-1\|_{L^6(\Qn)}\|\eta_n\nabla\eta_n\|_{L^3(\Qn)}\right)\\
    &+\frac{1}{2}\|\nabla\eta_n\|_{L^{\infty}(\Qn)}^2\|\sqrt{\rho^0}-1\|_{L^2(\Qn)}^2\\
    &\leq \int_{\R^d}\frac{1}{2}|\nabla\sqrt{\rho^0}|^2\dd x+n^{\frac{d-3}{3}}\|\nabla\sqrt{\rho^0}\|_{L^2(\Qn)}\|\sqrt{\rho}_0-1\|_{L^6(\Qn)}+\frac{1}{2}n^{-2}\|\sqrt{\rho^0}-1\|_{L^2(\Qn)}^2.
    \end{aligned}
\end{equation*}
We notice that
\begin{equation*}
    \limsup_{n\rightarrow\infty}\left(n^{\frac{d-3}{3}}\|\nabla\sqrt{\rho^0}\|_{L^2(\Qn)}\|\sqrt{\rho}_0-1\|_{L^6(\Qn)}\right)=0
\end{equation*}
Indeed, one has that $\sqrt{\rho^0}-1\in H^1(\R^d)$ and therefore
\begin{equation*}
    \limsup_{n\rightarrow \infty}\|\sqrt{\rho^0}-1\|_{L^p(\Qn)}=0,
\end{equation*}
for all $2\leq p\leq p^{\ast}$ as consequence of the decay of the tails of $L^p$-functions.
It is immediate to see,
\begin{equation*}
     \limsup_{n\rightarrow\infty}\left(\frac{1}{2}n^{-2}\|\sqrt{\rho^0}-1\|_{L^2(\Qn)}^2\right)=0
\end{equation*}
The dominated convergence Theorem yields that 
\begin{equation*}
        \lim_{n\rightarrow\infty}\int_{\T_n^d}\frac{1}{2}|\nabla\sqrt{\rho_{0}^n}|^2\dd x= \int_{\R^d}\frac{1}{2}|\nabla\sqrt{\rho^0}|^2\dd x
\end{equation*}

For \eqref{eq:renenergyQNS}, we observe that $\sqrt{\rho_n^0}$ converges pointwise to $\sqrt{\rho^0}$ a.e. on $\R^d$. Since, $F(\cdot)$ is convex on $\R_+$, it follows that $F(\rho_{0}^n)$ converges pointwise to $F(\rho^0)$ a. e. on $\R^d$. We aim to show the desired inequality and the convergence by introducing $G$ defined as 
\begin{equation*}
    G(x)= \begin{cases}
   C \qquad &\text{if} \quad x\in \{\sqrt{\rho^0}\leq \frac{1}{\sqrt{2}}\},\\
    C|\sqrt{\rho^0}-1|^2 \qquad &\text{if} \quad x\in \{\frac{1}{\sqrt{2}}\leq \sqrt{\rho^0}\leq 1\},\\
    F(\rho^0) \qquad &\text{if} \quad x\in \{\sqrt{\rho^0}\geq 1\},
    \end{cases}
\end{equation*}
We notice that $G:\R_+\rightarrow \R_+$ and $G\in L^1(\R^3)$. Splitting the integral, we have that the contribution coming from the domain $ \{\sqrt{\rho^0}\leq \frac{1}{\sqrt{2}}\}$ is finite since the volume of the set is bounded. The second contribution is bounded as $\sqrt{\rho^0}-1\in H^1(\R^d)$ and the last is bounded as $F(\rho^0)\in L^1(\R^d)$.\\
We claim that $F(\sqrt{\rho_n^0}^2)(x)\leq G(x)$ for all $x\in \T_n^d$ and $n\in \N$. Indeed, if $x\in \{\sqrt{\rho^0}\leq \frac{1}{\sqrt{2}}\}$, then $x\in \{0\leq\sqrt{\rho_n^0}\leq 1\}$ for all $n\in \N$. Thus, on the given set $F(\sqrt{\rho_n^0}^2)\leq F(0)=\gamma-1$. Next, if $x\in \{\frac{1}{\sqrt{2}}\leq \sqrt{\rho^0}\leq 1\}$ then $x\in \{\frac{1}{\sqrt{2}}\leq \sqrt{\rho_n^0}\leq 1\}$ for all $n\in \N$. Thus, on the prescribed domain
\begin{equation*}
    F(\sqrt{\rho_n}^2)\leq C \left|\sqrt{\rho_n^0}^2-1\right|^2=\left|(\sqrt{\rho^0}-1)^2\eta_n^2+2(\sqrt{\rho^0}-1)\eta_n\right|^2\leq C |\sqrt{\rho}_0-1|^2,
\end{equation*}
as $|\sqrt{\rho^0}-1|\leq 1$.
If $x\in \{\sqrt{\rho^0}\geq 1\}$, then in particular $x\in \{\sqrt{\rho_n^0}\geq 1\}$ for all $n\in \N$. Thus, the concatenation $F((\sqrt{\rho_n^0})^2)$ is a convex function on the set $\{\sqrt{\rho^0}\geq 1\}$ and 
\begin{equation*}
    F(\sqrt{\rho_n^0}^2)\leq \eta_n F(\sqrt{\rho^0}^2)+(1-\eta_n)F(1)\leq F(\rho^0).
\end{equation*}
We may therefore apply the dominated convergence Theorem to obtain
\begin{equation*}
\begin{aligned}
\lim_{n\rightarrow\infty}\int_{\T_n^d}F(\rho_n^0)\dd x=\int_{\R^d}F(\rho^0)\dd x.
\end{aligned}
\end{equation*} 
The convergence \eqref{eq:Lambda} is immediate. 
\end{proof}

\subsection{Existence of solutions on periodic domains}
In this section, we discuss the existence of a sequence of weak solutions to the system \eqref{eq:QNSFEWS} on $\T_n^d$ with initial data $(\sqrt{\rho_{0}^{n}},\sqrt{\rho_{0}^n}u_{0}^n)$ provided by Lemma \ref{lem:data}.  In \cite{LV16}, the authors show global existence of weak solutions to  \eqref{eq:QNSFEWS} posed on $[0,T)\times\T_n^d$ for $\gamma>1$, $\nu>0$ and $\kappa\geq 0$ complemented with initial data of finite energy. The construction of weak solutions proceeds in several steps: 
\begin{enumerate}
\item The weak solutions to an auxiliary system including drag forces provided in \cite{VY16} satisfy a truncated formulation of the equations.
\item Suitable compactness properties of the truncated solutions allow one to construct a truncated solution to \eqref{eq:QNSFEWS}.
\item The truncated weak solutions to \eqref{eq:QNSFEWS} are shown to be in particular weak solutions to \eqref{eq:QNSFEWS}.
\end{enumerate}
By a different approach, namely constructing smooth approximate weak solutions to \eqref{eq:QNSFEWS}, the existence of weak solutions to \eqref{eq:QNSFEWS} posed on $\T_n^d$ satisfying properties (i)-(iii) has also been obtained in \cite{AS, LZZ}, following the strategy of \cite{LX}. These results require additional restrictions on $\gamma,\kappa$. The solutions obtained in \cite{AS} can be shown to be finite energy weak solutions to \eqref{eq:QNSFEWS} on $\T_n^d$ in the sense of Definition \ref{defi:FEWS}, namely the solutions provided in \cite{AS} satisfy \eqref{eq:energyQNS} and \eqref{ineq:BDEQNS} with $C=1$, see Appendix A of \cite{AHM18}. Following the arguments in \cite{LV16} and in the Appendix A of \cite{AHM18}, it can be checked that the solutions of \cite{LV16} also enter the class of  finite energy weak solution in the sense of Definition \ref{defi:FEWS}.
We define 
\begin{equation}\label{eq:kappan}
  \kappa_n=\begin{cases}
  \kappa_n=\kappa \qquad \text{if} \quad \kappa>0,\\
  \kappa_n=\frac{1}{n} \qquad \text{if} \quad \kappa=0.
  \end{cases}  
\end{equation}

The notion of (approximate) truncated weak solution to \eqref{eq:QNSFEWS} is the following.

\begin{definition}\label{defi:truncatedsolution}
Let $d=2,3$. Given a domain $\Omega$, a sequence
$(\sqrt{\rho_n},\sqrt{\rho_n}u_n)$ is called approximate truncated weak solution to \eqref{eq:QNSFEWS} on $(0,T)\times\Omega$ with initial data $(\rho_n^0,u_n^0)$ if the following are satisfied.
\begin{enumerate}
    \item Integrability conditions
\begin{equation*}
\begin{aligned}
&\sqrt{\rho_n}\in L_{loc}^{2}((0,T)\times \Omega); \quad 
\sqrt{\rho_n}u_n\in L_{loc}^{2}((0,T)\times \Omega); \quad  \quad \nabla\sqrt{\rho_n}\in L_{loc}^{2}((0,T)\times \Omega); \\
&\nabla\rho_n^{\frac{\gamma}{2}}\in L_{loc}^{2}((0,T)\times \Omega); \quad \Tvn\in L_{loc}^2((0,T)\times \Omega); \quad \kappa_n\nabla^2\sqrt{\rho_n}\in L_{loc}^{2}((0,T)\times \Omega);\\
&\sqrt{\kappa_n}\nabla\rho_n^{\frac{1}{4}}\in L_{loc}^4((0,T)\times\Omega);
\end{aligned}
\end{equation*}
\item (Approximate continuity equation) there exists a sequence of distributions $D_n\in \mathcal{D}'([0,T)\times \Omega)$ such that $D_n\rightarrow 0$ in $\mathcal{D}'$ as $n\rightarrow \infty$ and
\begin{equation}\label{eq:approxCE}
    \int_{\Omega}\rho^0\psi(0,x) \dd x+\int_{\Omega}\rho_n\psi_t+\rho_nu_n\cdot \nabla\psi\dd x \dd t=\left\langle D_n,\psi\right\rangle,
\end{equation}
for any $\psi\in C_c^{\infty}([0,T)\times \Omega)$.
\item (Approximate momentum equation and compatibility conditions) for  any truncation function $\beta$ as in Definition \ref{defi:truncation}, there exist sequences of measures $\mu_{\beta}^n, \overline{\mu}_{\beta}^n$ obeying the bound
\begin{equation*}
     \|\mu_{\beta}^n\|_{\mathcal{M}}+\|\overline{\mu}_{\beta}^n\|_{\mathcal{M}}\leq C\|\nabla^2\beta\|_{L^{\infty}},
\end{equation*}
uniformly in $n$ and distributions $G_n, K_n, V_n \in \mathcal{D}'([0,T)\times\Omega)$ such that 
$G_n, K_n,V_n \rightarrow 0$ in $\mathcal{D}'$  as $n\rightarrow \infty$
and such that
\begin{equation}\label{eq:MEtrunc}
     \begin{aligned}
    &\int_{\Omega}\rho^0 \bdl(u^0)\psi(0,x)\dd x+\int_0^T\int_{\Omega}\rho_n\bdl(u_n)\partial_t\psi+\rho_n\bdl(u_n)u_n\cdot \nabla\psi \dd x \dd t\\
    &-2\int_0^T\int_{\Omega}\left(\sqrt{\nu}\Svn\nyb(u_n)+\kappa_n\sqrt{\rho_n}\Skn\nyb(u_n))\cdot\nabla\psi
    +2\rho_n^{\frac{\gamma}{2}}\nabla\rho_n^{\frac{\gamma}{2}}\right)\nyb(u)\cdot\nabla\psi\dd x\dd t\\
    &=\left\langle \mu_{\beta}^n+G_n,\psi\right\rangle,
    \end{aligned}
\end{equation}
with $\Svn$ being the symmetric part of $\Tvn$ satisfying the compatibility condition
\begin{equation*}
    \sqrt{\nu}\sqrt{\rho_n}\nyb(u_n)_i[\Tvn]_{ik}=\nu\partial_j(\rho\nyb(u_n)_i u_{n,k})-2\nu\sqrt{\rho_n}u_{n,k}\nyb(u_n)_i\partial_j\sqrt{\rho_n}+\overline{\mu}_{\beta}^n+V_n
\end{equation*}
in $\mathcal{D}'([0,T)\times\Omega)$ and the capillary tensor $\Skn$ satisfying
\begin{equation*}
   \kappa_n\Skn=\kappa_n^2\left(\sqrt{\rho_n}\left(\nabla^2\sqrt{\rho_n}-4(\nabla\rho_n^{\frac{1}{4}}\otimes\nabla\rho_n^{\frac{1}{4}}\right)\right)+K_n.
\end{equation*}
in $\mathcal{D}'([0,T)\times\Omega)$.
\end{enumerate}
Finally, we say that $(\sqrt{\rho_n},\sqrt{\rho_n}u_n)$ is a sequence of truncated weak solutions if $D_n=G_n=K_n=V_n=0$. Further, a truncated weak solution is called finite energy truncated weak solution if in addition \eqref{eq:energyQNS} and \eqref{ineq:BDEQNS} are satisfied.
\end{definition}

The following global existence Theorem for weak solutions in the sense of Definition \ref{defi:truncatedsolution} holds, see \cite{LV16}.
\begin{theorem}\label{thm:FEWStorus}
Let $\gamma>1, \nu>0$, $\kappa_n$ as defined in \eqref{eq:kappan} and $(\sqrt{\rho_n^0},\sqrt{\rho_n^0}u_n^0)$ be provided by Lemma \ref{lem:data}. Then there exists a sequence  $(\sqrt{\rho_n},\sqrt{\rho_{n}}u_{n})$ of finite energy truncated weak solutions to \eqref{eq:QNSFEWS} on $(0,T)\times\T_n^d$ with measures $R_\beta^n$, $\overline{R}_\beta^n$. In particular, $(\sqrt{\rho_n},\sqrt{\rho_{n}}u_{n})$ defines a finite energy weak solution to \eqref{eq:QNSFEWS} on $[0,T)\times\T_n^d$.
\end{theorem}

Several remarks are in order. 
\begin{enumerate}
    \item In \eqref{eq:energyQNS} and \eqref{ineq:BDEQNS} the pressure term $\frac{1}{\gamma-1}\rho^{\gamma}$, has been replaced by the internal energy $F(\rho)$. We recall that the internal energy and the pressure law are related by the equation $p(\rho)=F'(\rho)\rho-F(\rho)$. On a bounded domain $\Omega$, one easily checks that $F(\rho_n)\in L^1(\Omega)$ is equivalent to $\rho_n\in L^{\gamma}(\Omega)$.
    \item The bounds on the measures $R_{\beta}^n,\overline{R}_{\beta}^n$ are uniform in $n$ since only depending on the second derivatives of $\beta$ being bounded in virtue of Lemma \ref{lem:truncation}.
\item We comment on the energy and entropy inequalities. The solutions provided by Theorem \ref{thm:FEWStorus} satisfy the following bounds uniformly in $n$. Since the solutions provided by Theorem \ref{thm:FEWStorus} satisfy the energy \eqref{eq:energyQNS} and BD entropy inequality \eqref{ineq:BDEQNS} and we infer from Lemma \ref{lem:data} that
\begin{equation}\label{eq:energyuniform}
\begin{aligned}
   & \limsup_{n\rightarrow\infty}\left(E_{\T_n^d}(\sqrt{\rho_n},\sqrt{\rho_n}u_n)(t)+\int_0^T\int_{\T_n^d}|\Svn|^2\dd x \dd t\right)
    \leq \limsup_{n\rightarrow\infty}E_{\T_n^d}(\sqrt{\rho_n^0},\sqrt{\rho_n^0}u_n^0)\\
   & \leq C \int_{\R^d}\frac{1}{2}\rho^0|u^0|^2+2\kappa^2|\nabla\sqrt{\rho^0}|^2+F(\rho^0)\dd x.
    \end{aligned}
\end{equation}
From Theorem \ref{thm:FEWStorus} and again from Lemma \ref{lem:data} we conclude that there exists $C>0$ such that 
\begin{equation}\label{eq:BDuniform}
\begin{aligned}
    &\limsup_{n\rightarrow\infty}\left( {B}_{\T_n^d}(t)+\int_0^T\left(\nu|\nabla\rho_n^{\frac{\gamma}{2}}|^2+\nu\kappa_n^2(|\nabla\rho_n^{\frac{1}{4}}|^4+|\nabla^2\sqrt{\rho_n}|^2)+|\Tvn|^2\right)\dd x\dd t\right)\\
    &\leq  C \limsup_{n\rightarrow\infty} \left(1+ B_{\T_n^d}(\sqrt{\rho_n^0},\sqrt{\rho_n^0}u_n^0)\right)\\
    &\leq C\int_{\R^d}\frac{1}{2}\rho^0|u^0|^2+(2\kappa^2+4v^2)|\nabla\sqrt{\rho^0}|^2+F(\rho^0)\dd x.
    \end{aligned}
\end{equation}
\end{enumerate}

\section{Extension to approximate solutions and convergence}\label{sec:ext}
In this section, we show that there exists a finite energy truncated weak solution to \eqref{eq:QNSFEWS} on the whole space with far-field condition \eqref{eq:farfield}. The strategy of our method consists in several steps.
\begin{enumerate}[(i)]
    \item We extend the sequence of periodic solutions provided by Theorem \ref{thm:FEWStorus} to a sequence of functions defined on the whole space. 
\item We prove that the extensions provide a sequence of approximate truncated solutions to \eqref{eq:QNSFEWS} on $\R^d$ according to Definition \ref{defi:truncatedsolution}.
\item As $n$ goes to $\infty$, the sequence converges to a finite energy truncated weak solution to \eqref{eq:QNSFEWS}.
\end{enumerate}

Given a sequence of approximate truncated solutions on the invading domains $\T_n^d$ provided by Theorem \ref{thm:FEWStorus}, we define the density and momenta $(\rho_n:=(\sqrt{\rho_n})^2,m_n:=\sqrt{\rho_n}\sqrt{\rho_n}u_n)$. We extend $(\rho_n,m_n)$ by the stationary solution $(\rho=1,m=0)$ on $\R^d\backslash[-n,n]^d$ . Let $\eta_n\in C_c^{\infty}(\R^d)$ be defined as in \eqref{eq:eta} and denote $\Qn=[-n,n]^d\backslash (-\frac{n}{2},\frac{n}{2})^d$ so that $supp(\nabla\eta_n)\subset \Qn$. We introduce,
\begin{equation}\label{eq:extension}
\begin{aligned}
    &\widetilde{\rho_n}:=\rho_n\eta_n+(1-\eta_n),  
    \quad \tilde{m}_n=m_n\eta_n\\
    &\widetilde{\Svn}=\Svn\eta_n, \quad \widetilde{\Tvn}=\Tvn\eta_n.
    \end{aligned}
\end{equation}
Further, we denote 
\begin{equation*}
    \tilde{u}_n=\begin{cases}
    \frac{\tilde{m}_n(t,x)}{\tilde{\rho}_n(t,x)} \qquad &\text{if} \quad (t,x)\in\{\tilde{\rho}_n>0\},\\
    0 \qquad &\text{if} \quad (t,x)\in \{\tilde{\rho}_n=0\}.
    \end{cases}
\end{equation*}
Finally, we define
\begin{equation}\label{eq:datatilde}
    \tilde{\rho}_{n}^0=(\sqrt{\rho_n^0})^2\eta_n+(1-\eta_n), \qquad \tilde{m}_n^0=m_n^0\eta_n,
\end{equation}
and 
\begin{equation*}
    \tilde{u}_n^0=\begin{cases}
     \frac{\tilde{m}_n^0(x)}{\tilde{\rho}_n^0(x)} \qquad &\text{if} \quad x\in\{\tilde{\rho}_n^0>0\},\\
    0 \qquad &\text{if} \quad x\in \{\tilde{\rho}_n^0=0\}.
    \end{cases}
\end{equation*}

The main result of this Section is the following.
\begin{theorem}\label{thm:extensions}
Let $d=2,3$, $\gamma>1$, $\nu>0$, $\kappa\geq 0$ and $\kappa_n$ as defined in \eqref{eq:kappan}. Then $(\tilde{\rho}_n,\tilde{u}_n)$ defined in \eqref{eq:extension} is an approximate truncated weak solution to \eqref{eq:QNSFEWS} with initial data $(\tilde{\rho}_n^0,\tilde{u}_n^0)$ given by \eqref{eq:datatilde} and viscosity and capillary tensor $\widetilde{\Tvn}$ and $\widetilde{\Skn}$ respectively. Further, the measures $\mu_\beta^n, \overline{\mu_\beta^n}$ satisfy
\begin{equation*}
    \mu_\beta^n=R_\beta^n\eta_n, \quad \overline{\mu_\beta^n}=\overline{R}_\beta^n\eta_n,
\end{equation*}
with ${R}_\beta^n, \overline{R}_\beta^n$ provided by Theorem \ref{thm:FEWStorus}.\\
Moreover, as $n$ goes to infinity, the sequence $(\tilde{\rho}_n,\tilde{u}_n,\widetilde\Tvn,\widetilde\Skn)$ converges to a finite energy truncated weak solution $(\rho,u)$ with initial data $(\rho^0,u^0)$ and with viscosity and capillary tensors $(\Tv,\Sk)$ weak $L^2$-limits of $\widetilde\Tvn, \widetilde\Svn$ respectively. More precisely, $\Tilde{\rho}_n$ converges strongly to $\rho$ in $C(\R_+;L_{loc}^p(\R^d))$ for $1<p<\sup\{3,\gamma\}$, the momenta $\tilde{m}_n$ converge strongly to $m$ in $L_{loc}^{2}(\R_+;L_{loc}^p(\R^d))$ in $1\leq p<\frac{3}{2}$ and $\tilde{\rho}_n\tilde{u}_n \bdl(\tilde{u}_n)$ converges strongly to $\rho u \bdl(u)$ in  $L_{loc}^{2}(\R_+;L_{loc}^2(\R^d)$. 
\end{theorem}

We notice that the measures are well-defined on $(0,T)\times \R^d$ taking into account the support properties of $\eta_n$.
We start by collecting the needed uniform estimates that will follow from \eqref{eq:energyuniform} as well as \eqref{eq:BDuniform}. These will be used to show the first part of Theorem \ref{thm:extensions}. Subsequently, we prove suitable compactness properties for the sequence $(\Tilde{\rho}_n,\Tilde{u}_n)$ needed for the passage to the limit as $n\rightarrow\infty$.

\begin{lemma}\label{lem:extensions}
The extensions introduced in \eqref{eq:extension} obey the following bounds uniformly in $n$,
\begin{equation}\label{eq:uniformextension}
    \begin{aligned}
    &\tilde{\rho}_n-1\in L^{\infty}(\R_+;L^2(\R^d)), \quad  F(\tilde\rho_n)\in L_{loc}^{\infty}(\R_{+};L^1(\R^d)),\\
    &{\sqrt{\tilde{\rho}_n}}-1\in L_{loc}^{\infty}(\R_{+};H^1(\R^d)), \quad \tilde{m}_n\in L^{\infty}(\R_+,L^{\frac{3}{2}}(\R^d)+L^2(\R^d)),\\
    &\sqrt{\tilde{\rho}_n}\tilde{u}_n\in L_{loc}^{\infty}(\R_{+};L^2(\R^d)), \quad \nabla\tilde{\rho}_n^{\frac{\gamma}{2}}\in L_{loc}^{2}(\R_{+};L^2(\R^d)), \quad \widetilde{\Tvn}\in L_{loc}^{2}(\R_{+};L^2(\R^d)).
    \end{aligned}
\end{equation}
Moreover, there exists $\sqrt{\rho}, m, \Tv, \Sv, \Sk$ such that
\begin{equation}\label{eq:extensionwc}
    \begin{aligned}
       &\sqrt{\tilde{\rho}_n}-1\rightharpoonup\sqrt{\rho}-1 \quad \text{in} \quad L^{\infty}(\R_+;H^1(\R^d)), \quad
       \tilde{m}_n\rightharpoonup m \quad \text{in} \quad L^{\infty}(\R_+;L^{\frac{3}{2}}+L^2(\R^d)),\\
       &\widetilde{\Tvn}\rightharpoonup \Tv \quad \text{in} \quad  L_{loc}^{2}(\R_{+};L^2(\R^d)), \quad
       \widetilde{\Svn}\rightharpoonup \Sv \quad \text{in} \quad  L_{loc}^{2}(\R_{+};L^2(\R^d)).
    \end{aligned}
\end{equation}
If $\kappa>0$, one has additionally that uniformly in $n$, 
\begin{equation}\label{eq:uniformextensionk}
    \kappa_n\nabla^2\sqrt{\tilde{\rho}_n}\in L_{loc}^{2}(\R_{+};L^2(\R^d)), \quad  \kappa_n^{\frac{1}{4}}\nabla\tilde{\rho}_n^{\frac{1}{4}}\in L_{loc}^{4}(\R_{+};L^4(\R^d)),
\end{equation}
and $\widetilde\Skn,\Sk\in L_{loc}^2(\R_{+};L^2(\R^d))$ such that $ \widetilde{\Skn}\rightharpoonup \Sk$ in $ L_{loc}^2(\R_{+};L^2(\R^d))$.
Finally, for a.e. $t\in [0,T)$, one has that
\begin{equation}\label{eq:energyextension}
    \begin{aligned}
        &\limsup_{n\rightarrow\infty}\Big(\int_{R^d}\frac{1}{2}\tilde{\rho}_n|\tilde{u}_n|^2+2\kappa_n^2|\nabla\sqrt{\tilde{\rho}_n}|^2+F(\tilde{\rho}_n)\dd x+2\nu\int_0^T\int_{\R^d}|\widetilde\Svn|^2\dd x\dd t\Big)\\
        &\leq C \int_{\R^d}\frac{1}{2}\rho^0|u^0|^2+2\kappa^2|\nabla\sqrt{\rho^0}|^2\dd x,
    \end{aligned}
\end{equation}
and 
\begin{equation}\label{eq:BDextensions}
\begin{aligned}
   & \limsup_{n\rightarrow\infty}\Big(\int_{\R^d}\frac{1}{2}\left|\sqrt{\tilde{\rho}_n}\tilde{u}_n(t)\right|^2+(\kappa_n^2+2\nu^2)|\nabla{\sqrt{\tilde{\rho}_n}}(t)|^2 +F((\tilde{\rho}_n)^2)(t)\dd x\\
    &+2\nu\int_0^T\int_{\R^d}|\widetilde{\Tvn}|^2\dd x\dd t+\nu\int_{0}^t\int_{\R^d}|\tilde{\rho}_n^{\frac{\gamma}{2}}|^2\dd x \dd t +\nu\kappa_n\int_0^T\int_{\R^d}|\nabla\tilde{\rho}_n^{\frac{1}{4}}|^{4}+|\nabla^2\sqrt{\tilde{\rho}_n}|^2\Big)\\
    &\leq C\left(1+ \int_{\R^d}\frac{1}{2}\rho^0|u^0|^2+(2\kappa^2+4v^2)|\nabla\sqrt{\rho^0}|^2+F(\rho^0)\dd x\right).
    \end{aligned}
\end{equation}
\end{lemma}

\begin{proof}
We start by showing the bounds \eqref{eq:uniformextension}.
One has
\begin{equation*}
    \limsup_{n\rightarrow\infty}\|\tilde\rho_n-1\|_{L_t^{\infty}L_x^2(\R^d)}= \limsup_{n\rightarrow\infty}\|(\rho_n-1)\eta_n\|_{L_t^{\infty}L_x^2(\R^d)}\leq C \limsup_{n\rightarrow\infty}\|(\rho_n-1)\|_{L_t^{\infty}L_x^2(\T_n^d)},
\end{equation*}
that is bounded in view of \eqref{eq:energyuniform} and \eqref{eq:BDuniform}. Indeed, proceeding as in the proof of Lemma \ref{lem:dataQNS}, one obtains that  the right hand side is bounded by the sum of $F(\rho_n)\in L^{\infty}(\R_+;L^1(\T_n^d))$ and $\nabla\sqrt{\rho_n}\in L^{\infty}(\R_+;L^2(\T_n^d)))$. Those, in their turn, are uniformly bounded in virtue of \eqref{eq:energyuniform} and \eqref{eq:BDuniform}. Next, due to convexity of the renormalized internal energy and $F(1)=0$, one has
\begin{equation*}
    F(\tilde{\rho}_n)\leq \eta_nF(\rho_n),
\end{equation*}
which yields
\begin{equation*}
    \limsup_{n\rightarrow\infty}\|F(\tilde{\rho}_n)\|_{L^{\infty}(\R_+;L^1(\R^d))}\leq \limsup_{n\rightarrow\infty} \|F(\rho_n)\|_{L^{\infty}(\R_+;L^1(\T_n^d))},
\end{equation*}
being bounded again by \eqref{eq:energyuniform}. The pointwise inequality
\begin{equation*}
    |\sqrt{\tilde{\rho}_n}-1|\leq |\tilde{\rho}_n-1|,
\end{equation*}
yields the bound $\sqrt{\tilde{\rho}_n}-1\in L^{\infty}(\R_+;L^2(\R^d))$ uniformly. Next,
\begin{equation*}
\begin{aligned}
&\limsup_{n\rightarrow\infty} \|\nabla\sqrt{\tilde{\rho}_n}\|_{L^{\infty}(\R_+;L^2(\R^d))}\\
&=\limsup_{n\rightarrow\infty} \left\|\frac{1}{2\sqrt{\rho_n\eta_n+(1-\eta_n)}}\left(\eta_n\nabla\rho_n+(\rho_n-1)\nabla\eta_n\right)\right\|_{L^{\infty}(\R_+;L^2(\R^d))}\\
&\leq C\limsup_{n\rightarrow\infty}\left( \left\|\frac{\sqrt{\eta_n}}{2\sqrt{\rho_n}}\nabla\rho_n\right\|_{L^{\infty}(\R_+;L^2(\R^d)}+\left\|\frac{(\rho_n-1)}{2\sqrt{(\rho_n-1)\eta_n+1}}\nabla\eta_n\right\|_{L^{\infty}(\R_+;L^2(\R^d))}\right)\\
&\leq C\limsup_{n\rightarrow\infty}\left\|\nabla\sqrt{\rho_n}\right\|_{L^{\infty}(\R_+;L^2(\T_n^d))},
\end{aligned}
\end{equation*}
where we used that
\begin{equation*}
    \limsup_{n\rightarrow\infty}\left\|\frac{\rho_n-1}{\sqrt{(\rho_n-1)\eta_n+1}}\nabla\eta_n\right\|_{L^{\infty}(\R_+;L^2(\T_n^d))}\leq C \limsup_{n\rightarrow\infty}\|\rho_n-1\|_{L^{\infty}(\R_+;L^2(\Qn))}=0
\end{equation*}
following from the integrability properties of $\rho_n-1$, the $L^{\infty}$-bound  and the support properties for the cut-off $\eta_n$ and its gradient. The bound on the momenta is an immediate consequence of  \eqref{eq:energyuniform} by observing that $|\tilde{m}_n|\leq |m_n|$ on $[-n,n]^d$ and $\tilde{m}_n=0$ on $\R^d\backslash[-n,n]^d$. The bound on $\widetilde{\Tvn}$ is analogous. 
Next, we show the bound on $\nabla\tilde{\rho}_n^{\frac{\gamma}{2}}$. If $\gamma\geq 3$, then $f(t)=t^{\frac{\gamma-1}{2}}$ is convex and therefore
\begin{equation*}
    \begin{aligned}
       &\limsup_{n\rightarrow\infty}\|\nabla\tilde{\rho}_n^{\frac{\gamma}{2}}\|_{L^2(0,T;L^2(\R^d))}=\limsup_{n\rightarrow\infty}\|\gamma(\sqrt{\tilde{\rho}_n})^{\gamma-1}\nabla\sqrt{\tilde{\rho}_n}\|_{L^2(0,T;L^2(\R^d))}\\
       &\leq\limsup_{n\rightarrow\infty}\|\gamma\left((\eta_n\rho_n^{\frac{\gamma-1}{2}}+(1-\eta_n))\nabla\sqrt{\tilde{\rho}_n}\right)\|_{L^2(0,T;L^2(\T_n^d))}, 
    \end{aligned}
\end{equation*}
and proceeding analogously as in the bound for $\nabla\sqrt{\tilde{\rho}_n}$ we conclude by invoking \eqref{eq:BDuniform}. If $1<\gamma<3$, we use that $f(t)=t^{\frac{\gamma-1}{2}}$ is a concave function s.t. $f(0)=0$ and therefore sub-additive and proceed as in the previous case. We conclude that 
\begin{equation*}
    \limsup_{n\rightarrow\infty}\|\nabla\tilde{\rho}_n^{\frac{\gamma}{2}}\|_{L^2(0,T;L^2(\R^d))}\leq C \limsup_{n\rightarrow\infty}\|\nabla{\rho_n}^{\frac{\gamma}{2}}\|_{L^2(0,T;L_{x}^2(\T_n^d))}.
\end{equation*}\\
Finally, we show the bounds \eqref{eq:uniformextensionk},
\begin{equation*}
    \begin{aligned}
       &\limsup_{n\rightarrow\infty}\|\nabla\tilde{\rho}_n^{\frac{1}{4}}\|_{L^4(0,T;L^4(\R^d))}\\
       &\leq C  \limsup_{n\rightarrow\infty}\left(\left\|\frac{\eta_n}{4(\rho_n\eta_n+1-\eta_n)^{\frac{3}{4}}}\nabla\rho_n\right\|_{L_t^4L_x^4(\R^d)}+\left\|\frac{(\rho_n-1)}{4(\rho_n\eta_n+1-\eta_n)^{\frac{3}{4}}}\nabla\eta_n\right\|_{L_t^4L_x^4(\R^d)}\right)
    \end{aligned}
\end{equation*}
The first term is controlled by
\begin{equation*}
    \limsup_{n\rightarrow\infty}\left\|\frac{\eta_n}{4(\rho_n\eta_n+1-\eta_n)^{\frac{3}{4}}}\nabla\rho_n\right\|_{L_t^4L_x^4(\R^d)}\leq\left\|\eta_{n}^{\frac{1}{4}}\nabla\rho_n^{\frac{1}{4}}\right\|_{L_t^4L_x^4(\T_n^d)},
\end{equation*}
that is uniformly bounded from \eqref{eq:BDuniform}. 
The second term is estimated as 
\begin{equation*}
\begin{aligned}
   &\limsup_{n\rightarrow\infty} \left\|\frac{(\rho_n-1)}{4(\rho_n\eta_n+1-\eta_n)^{\frac{3}{4}}}\nabla\eta_n\right\|_{L_t^4L_x^4(\R^d)}\\
    &\leq \limsup_{n\rightarrow\infty}\left(\|\mathbf{1}_{\{\rho_n-1<0\}}\frac{1}{4(1-\eta_n)^{\frac{3}{4}}}\nabla\eta_n\|_{L_t^4L_x^4(\Qn)}+ \|\mathbf{1}_{\{\rho_n-1>0\}}\frac{(\rho_n-1)^{\frac{1}{4}}}{4\eta^{\frac34}}\nabla\eta_n\|_{L_t^4L_x^4(\Qn)}\right)\\
&\leq C \left(1+\limsup_{n\rightarrow\infty}(\|\rho_n-1\|_{L_t^4L_x^2(\Qn)}\right)=C.
    \end{aligned}
\end{equation*}
Thus,
\begin{equation*}
    \limsup_{n\rightarrow\infty}\|\nabla\tilde{\rho}_n^{\frac{1}{4}}\|_{L^4(0,T;L^4(\R^d))}\leq C  \limsup_{n\rightarrow\infty}\|\nabla\tilde{\rho}_n^{\frac{1}{4}}\|_{L^4(0,T;L^4(\T_n^d))}.
\end{equation*}
It remains to bound $\nabla^2\sqrt{\tilde{\rho}_n}$ in $L^2(0,T;L^2(\R^d))$. To that end we compute that 
\begin{equation*}
\begin{aligned}
    \nabla^2\sqrt{\tilde{\rho}_n}&=\frac{1}{2\sqrt{\tilde{\rho}_n}}\eta_n\nabla^2\rho_n-\frac{1}{4}\frac{\eta_n^2}{\tilde{\rho}_n^\frac{3}{2}}(\nabla\rho_n)^2\\
    &+\frac{1}{2}\frac{\rho_n}{\sqrt{\tilde{\rho}_n}}\nabla^2\eta_n-\frac{1}{4}\frac{(\rho_n-1)^2}{\tilde{\rho}_n^\frac{3}{2}}(\nabla\eta_n)^2\\
    &+\frac{1}{\sqrt{\tilde{\rho}_n}}\nabla\eta_n\nabla\rho_n-\frac{1}{2}\frac{\eta_n(\rho_n-1)}{\tilde{\rho}_n^\frac{3}{2}}\nabla\eta_n\nabla\rho_n.
    \end{aligned}
\end{equation*}
Proceeding analogously as for the previous bound, the $L_t^2L_x^2$-norm of the RHS of the first line is bounded by $\|\nabla^2\sqrt{\rho_n}\|_{L^2(0,T;L^2(\T_n^d))}$ that again is uniformly bounded in view of \eqref{eq:energyuniform}. The other terms can be controlled by exploiting the properties of $\eta_n$ and \eqref{eq:energyuniform} and 
\eqref{eq:BDuniform} holding uniformly in $n$.
Finally, we observe that since
\begin{equation*}
    \left|\sqrt{\tilde{\rho}_n}\tilde{u}_n\right|= \left|\frac{m_n\eta_n}{\sqrt{\rho_n\eta_n+(1-\eta_n)}}\right|\leq \sqrt{\eta_n}\left|\frac{m_n}{\sqrt{\rho_n}}\right|,
\end{equation*}
one has
\begin{equation*}
\begin{aligned}
    \int_{\R^d}\frac{1}{2}\tilde{\rho}_n|\tilde{u}_n|^2\dd x\leq \int_{\T_n^d}\frac{1}{2}\sqrt{\rho_n}\left|u_n\right|^2\dd x
    \end{aligned}
\end{equation*}

Therefore, combing the previous inequalities with \eqref{eq:energyuniform} and \eqref{eq:BDuniform} we conclude that inequalities \eqref{eq:energyextension} and \eqref{eq:BDextensions} are satisfied for a.e. $t\in [0,T)$.
Hence, the uniform bounds \eqref{eq:uniformextension} follow.
\end{proof}

The uniform bounds lead to the following convergence results. 
\begin{lemma}\label{lem:convergenceQNS}
The following convergences hold up to subsequences. 
\begin{enumerate}
    \item $\tilde\rho_n\rightarrow \rho$ strongly in $C(\R_+;L_{loc}^p(\R^d))$ for $1<p<\sup\{3,\gamma\}$.
    \item $\tilde{m}_n\rightarrow m$ strongly in $L_{loc}^{2}(\R_+;L_{loc}^p(\R^d))$ in $1\leq p<\frac{3}{2}$.
     \item $\nabla\tilde{\rho}_n^{\frac{\gamma}{2}}\rightharpoonup \nabla\rho^{\frac{\gamma}{2}}$ weakly in $L^2(0,T;L_{loc}^2(\R^d))$.
\end{enumerate}
If $\kappa>0$, then
\begin{enumerate}
   \item $\sqrt{\tilde{\rho}_n}\rightarrow\sqrt{\rho}$ strongly in $L_{loc}^2(0,T;H_{loc}^1(\R^d))$,
   \item $\nabla^2\sqrt{\tilde{\rho}_n}\rightharpoonup\nabla^2\sqrt{\rho}$ weakly in $L_{loc}^2(0,T;L_{loc}^2\R^d))$.
\end{enumerate}
\end{lemma}

\begin{proof}
We repeatedly use that $(\sqrt{\rho_n}, \sqrt{\rho_n}u_n)$ is a finite energy weak solution to \eqref{eq:QNSFEWS} on the scaled torus. 
\begin{enumerate}
    \item One has that uniformly in $n$,
    \begin{equation*}
    \begin{aligned}
        &\|\partial_{t}\tilde{\rho}_n\|_{L^{\infty}(\R_+;W_{loc}^{-1,\frac{3}{2}}(\R^d))}= \|\eta_n\partial_{t}\rho_n\|_{L^{\infty}(\R_+;W_{loc}^{-1,\frac{3}{2}}(\R^d))}\\
       & \leq \|m_n\|_{L^{\infty}(\R_+;L_{loc}^{\frac{3}{2}}(\T_n^d))},
          \end{aligned}
    \end{equation*}
    being bounded in virtue of \eqref{eq:energyuniform}.
    On the other hand, 
   $\tilde{\rho}_n\in L^{\infty}(\R_+;L_{loc}^{\gamma}\cap L_{loc}^3(\R^d))$ and $\nabla\tilde{\rho}_n\in L^{2}(\R^d)+L^{\frac{3}{2}}(\R^d)$ and hence $\tilde{\rho}_n\in L^{\infty}(\R_+;W_{loc}^{1,\frac{3}{2}}(\R^d))$.
    Thus, we conclude from the Aubin-Lions Lemma that $\tilde{\rho}_n-1$ is compact in $C(\R_+;L_{loc}^p(\R^d))$
    for any $1\leq p<\sup\{3,\gamma\}$.
    \item Since $\partial_t\tilde{m}_n=\eta_n\partial_tm_n$, from the second equation of \eqref{eq:QNSFEWS} we infer that $\partial_t(\rho_nu_n)$ is uniformly bounded in $L^2(\R_+;H_{loc}^{-s}(\R^d)$ for $s$ large enough by applying the uniform bounds of \eqref{eq:energyuniform} and \eqref{eq:BDuniform}.
    From \eqref{eq:uniformextension}, we have that $\tilde{m}_n\in L^{\infty}(\R_+;L^\frac{3}{2}+L^2(\R^d)$ and 
    $\nabla\tilde{m}_n=\eta_n\nabla m_n+m_n\nabla\eta_n$ is bounded in $L^{2}(0,T;L_{loc}^1(\R^d))$, thus $\tilde{m}_n\in L^2(0,T;W_{loc}^{1,1}(\R^d))$. The Aubin-Lions lemma implies that $\tilde{m}_n$ is compact in $L^2(0,T;L_{loc}^{p}(\R^d))$ for $1\leq p<\frac{3}{2}$.
    \item The uniform bound $\nabla\tilde{\rho}_n^{\frac{\gamma}{2}}\in L^2(0,T;L^2(\R^d))$ implies that up to passing to subsequences the sequence converges weakly with the weak limit being identified with $\nabla\rho^{\frac{\gamma}{2}}$ by means of point (1).
\end{enumerate}
If $\kappa>0$, we have $\nabla^2\sqrt{\tilde{\rho}_n}\in L^2(0,T;L^2(\R^d)$ additionally to $\sqrt{\tilde{\rho}_n}-1\in L^{\infty}(\R_+;H^1(\R^d))$. Hence, $\nabla^2\sqrt{\tilde{\rho}_n}$ converges weakly to $\nabla^2\sqrt{\rho}$ in $L_{loc}^2(0,T;L_{loc}^2(\R^d))$ up to subsequences.
Further, by combining the strong convergence of $\tilde{\rho}_n$ and the bounds on the second order derivatives of $\sqrt{\tilde{\rho}_n}$ we obtain that 
\begin{equation*}
    \sqrt{\tilde{\rho}_n}\rightarrow \sqrt{\rho} \qquad \text{in} \quad L_{loc}^2(0,T;H_{loc}^1(\R^d)).
\end{equation*}
\end{proof}

\begin{lemma}\label{lem:convergenceQNSren}
Let  $f\in C\cap L^{\infty}(\R^d;\R)$ and let $(\tilde{\rho}_n,\tilde{u}_n)$ be as defined in \eqref{eq:extension} and 
\begin{equation}\label{eq:defu}
u:=\begin{cases}
\frac{m(t,x)}{{\rho(t,x)}}
\quad &(t,x)\in\{\rho>0\},\\
0 \quad &(t,x)\in\{\rho=0\}.
\end{cases}
\end{equation}
Then,
\begin{enumerate}
\item for any $0<\alpha<\frac{5\gamma}{3}$, one has $\tilde{\rho}_n^\alpha f(\tilde{u}_n)\rightarrow \rho^\alpha f(u)$ in $L^p((0,T)\times\R^d))$ with $1\leq p<\frac{5\gamma}{3\alpha}$,
\item $\tilde{\rho}_nu_n f(\tilde{u}_n)\rightarrow \rho u f(u)$ strongly in $L_{loc}^{2}(\R_+;L_{loc}^2(\R^d))$,
\end{enumerate}
\end{lemma}

\begin{proof}
We recall that (1) and (2) of Lemma \ref{lem:convergenceQNS} imply that 
$\tilde{\rho}_n$ converges to $\rho$  a.e. in $(0,T)\times \R^d$ and $\tilde{m}_n$ converges to $m$  a.e. in $(0,T)\times \R^d$. 
The Fatou Lemma implies that
\begin{equation*}
\begin{aligned}
    &\int_{\R_+}\int_{\R^d} \liminf_{n\rightarrow\infty} \frac{\tilde{m}_n^2}{\tilde{\rho}_n}\dd x \dd t\leq \liminf_{n\rightarrow\infty} \int_{\R_+}\int_{\R^d}\frac{\tilde{m}_n^2}{\tilde{\rho}_n}\dd x \dd t\\
    &\leq \liminf_{n\rightarrow\infty} \int_{\R_+}\int_{\R^d}\eta_n|\sqrt{\rho_n}u_n|^2\dd x \dd t\leq \liminf_{n\rightarrow\infty} \int_{\R_+}\int_{\T_n^d}|\sqrt{\rho_n}u_n|^2\dd x \dd t <+\infty,
    \end{aligned}
\end{equation*}
due to inequality \eqref{eq:energyuniform}. Hence $m=0$ on the null set of $\rho$ and by consequence
$\sqrt{\rho}u\in L^{\infty}(0,T;L^2(\R^d))$. Moreover, since $\sqrt{\tilde{\rho}_n}\tilde{u}_n$ is uniformly bounded in $L^{\infty}(0,T;L^2(\R^d))$ it converges weakly-$\ast$ to some limit function $\Lambda\in L^{\infty}(0,T;L^2(\R^d))$. By uniqueness of weak-limits and the a.e. convergence of $\tilde{\rho}_n,\tilde{m}_n$ we recover, $m=\rho u=\sqrt{\rho}\Lambda$.

We show (1). It easily follows that 
\begin{equation*}
    \limsup_{n\rightarrow \infty}{\{\tilde{\rho}_n=0\}}\subset \{\rho=0\}, \qquad \text{and} \qquad \{\rho>0\}\subset \liminf_{n\rightarrow \infty}{\{\tilde{\rho}_n>0\}}.
\end{equation*}
Then
\begin{equation*}
    \tilde{\rho}_n^{\alpha}f(\tilde{u}_n)\rightarrow\rho^{\alpha} f(u), \quad \text{a.e. in} \quad \{\rho>0\}.
\end{equation*}
Moreover, since $f\in L^{\infty}(\R^d;\R)$ and $\alpha>0$ we have that
\begin{equation*}
   \left|\tilde{\rho}_n^{\alpha}f(u_n)\right|\leq |\tilde{\rho}_n|^{\alpha}\|f\|_{L^{\infty}}\rightarrow 0 \quad \text{a.e. in } \quad \{\rho=0\}.
\end{equation*}
It follows, that $\tilde{\rho}_n^{\alpha}f(\tilde{u}_n)$ converges to $\rho^{\alpha} f(u)$ a.e. in $(0,T)\times \R^d$.
From \eqref{eq:uniformextension}, we have that $\tilde{\rho}_n^{\frac{\gamma}{2}}\in L^{\infty}(\R_+;L_{loc}^2(\R^d))\cap L^{2}(0,T;L_{loc}^6(\R^d))$ uniformly.
By interpolation $\tilde{\rho}_n^{\frac{\gamma}{2}}\in L^{\frac{10}{3}}(0,T;L_{loc}^{\frac{10}{3}}(\R^d))$. Together with Vitali's convergence theorm, this yields strong convergence of $\tilde{\rho}_n^{\alpha}f(\tilde{u}_n)$ in $L_{loc}^{p}((0,T)\times\R^d))$ for $0<\alpha<\frac{5\gamma}{3}$ and $1\leq p<\frac{5\gamma}{3\alpha}$.\\
Next, we show (2). We use again that $\tilde{\rho}_n$ and $\tilde{m}_n$ converge a.e. in $(0,T)\times \R^d$ and conclude that
\begin{equation*}
\begin{aligned}
   & \Tilde{\rho}_n\Tilde{u}_nf(\Tilde{u}_n)\rightarrow m f(u) \quad \text{a.e. in} \quad \{\rho>0\},\\
    &\left|\Tilde{\rho}_n\Tilde{u}_n f(\Tilde{u}_n)\right|\leq \left|\tilde{m}_n\right|\|f\|_{L^{\infty}}\rightarrow 0 \qquad \text{a.e. in}\quad \{\rho=0\}.
    \end{aligned}
\end{equation*}
Vitali's convergence theorem together with the uniform bounds from Lemma \ref{lem:extensions} yield the strong convergence in $L_{loc}^{2}(\R_+;L_{loc}^2(\R^d)$.
\end{proof}

We are now in position to proof Theorem \ref{thm:extensions}.
\begin{proof}
Let $(\tilde{\rho}_n^0,\tilde{m}_n^0)$ and $\tilde{u}_n^0$ be defined by \eqref{eq:datatilde} and $(\tilde{\rho}_n,\tilde{m}_n)$ be defined by \eqref{eq:extension}. The required uniform bounds are consequence of \eqref{eq:uniformextension} and \eqref{eq:uniformextensionk}.
We compute
\begin{equation*}
    \begin{aligned}
         &\int_{\R^d}\tilde{\rho}_n^0\psi(0,x) \dd x+\int_0^T\int_{\R^d}\tilde{\rho}_n\psi_t+\tilde{m}_n\cdot \nabla\psi\dd x \dd t\\
         &=\int_{\R^d}\rho_n^0\eta_n\psi\dd x+\int_0^T\int_{R^d}\rho_n\partial_t(\eta_n\psi)+m_n\cdot\nabla(\eta_n\psi)\dd x \dd t\\
         &+\int_{\R^d}(1-\eta_n)\psi(0,x)\dd x+\int_0^T\int_{\R^d}\partial_t\left((1-\eta_n)\psi\right)-m_n\psi\cdot \nabla\eta_n\dd x \dd t\\
         &=-\int_0^T\int_{\R^d}m_n\psi\cdot\nabla\eta_n\dd x \dd t,
    \end{aligned}
\end{equation*}
where we used that $\eta_n\psi\in C_c^{\infty}([0,T)\times\T_n^d)$ and thus the second line is the weak formulation of the continuity equation for $(\rho_n,m_n)$ on $\T_n^d$ for an admissible test-function. We denote 
\begin{equation}\label{eq:Fn}
    \left\langle D_n,\psi\right\rangle=-\int_0^T\int_{\R^d}m_n\psi\cdot\nabla\eta_n\dd x \dd t,
\end{equation}
and observe that $D_n$ is well-defined and uniformly bounded in $\mathcal{D}'([0,T)\times\R^d)$ since $m_n\in L^{\infty}(0,T;L^{\frac{3}{2}}(\T_n^d))$ and $\eta_n\in C_c^{\infty}(\R^d)$. Further, as $supp(\nabla\eta_n)\subset \Qn$, we conclude $supp(D_n)\subset \Qn$. This allows to infer that $D_n$ converges to $0$ in $\mathcal{D}'$ as it is uniformly bounded and for $n$ sufficiently large $supp(\psi)\cap supp(D_n)=\emptyset$.
We proceed to verify that $(\tilde{\rho}_n,\tilde{u}_n)$ is an approximate solution to the truncated formulation of the momentum equation. We repeatedly use that $\eta_n(\bdl(\tilde{u}_n)-\bdl(u_n))\neq 0$ only on $\Qn$ and $supp(\nabla\eta_n)\subset \Qn$.
One has that,
\begin{equation}\label{eq:approxME0}
\begin{aligned}
    &\int_{\R^d}\tilde{\rho}_n^0\bdl(\tilde{u}_n^0)\psi(0,x)\dd x=\int_{\R^d}\rho_n^0\bdl(u_n^0)(\eta_n\psi_0)\dd x\\
    &+\int_{\Qn}\rho_n^0\left(\bdl(\tilde{u}_n^0)-\bdl(u_n^0)\right)\eta_n\psi(0,x)+(1-\eta_n)\psi(0,x)\dd x.
    \end{aligned}
\end{equation}
Next, one has
\begin{equation}\label{eq:approxME1}
\begin{aligned}
&\int_0^T\int_{\R^d}\tilde{\rho}_n\bdl(\tilde{u}_n)\partial_t\psi\dd x \dd t=\int_0^T\int_{\R^d} \rho_n\bdl(u_n)\partial_t(\eta_n\psi) \dd x \dd t\\
&+\int_0^T\int_{\Qn}\rho_n\left(\bdl(\tilde{u}_n)-\bdl(u_n)\right)\partial_t(\eta_n\psi)\dd x \dd t+(1-\eta_n)\bdl(\tilde{u}_n)\partial_t\psi(0,x)\dd x \dd t.
\end{aligned}
\end{equation}
Similarly,
\begin{equation}\label{eq:approxME2}
    \begin{aligned}
        &\int_0^T\int_{\R^d}\tilde{\rho}_n\tilde{u}_n\bdl(\tilde{u}_n)\cdot\nabla\psi\dd x\dd t
        =\int_0^T\int_{\R^d}\rho_n u_n\bdl(u_n)\cdot\nabla(\psi\eta_n)\dd x\dd t\\
        &+\int_0^T\int_{\Qn}\rho_n u_n\left(\left(\bdl(\tilde{u}_n)-\bdl(u_n)\right)\nabla(\psi\eta_n)-\bdl(\tilde{u}_n)\psi\nabla\eta_n\right)\dd x \dd t.
    \end{aligned}
\end{equation}
The pressure term is dealt with as follows,
\begin{equation}\label{eq:approxME3}
    \begin{aligned}
        &\int_0^T\int_{\R^d}2\tilde{\rho}_n^{\frac{\gamma}{2}}\nabla\tilde{\rho}_n^{\frac{\gamma}{2}}\nyb(\tilde{u}_n)\psi \dd x\dd t=\int_0^T\int_{\R^d}2\rho_n^{\frac{\gamma}{2}}\nabla\rho_n^{\frac{\gamma}{2}}\nyb(u_n)\eta_n\psi \dd x\dd t\\
        &-\int_0^T\int_{\R^d}\left(\tilde{\rho}_n^{\gamma}-\eta_n\rho_n^{\gamma}\right)\nabla(\nyb(\tilde{u}_n)\psi)\dd x \dd t\\
        &+\int_0^T\int_{\Qn}\rho_n^{\gamma}\nyb(\tilde{u}_n)\psi\nabla\eta_n+2\rho_n^{\frac{\gamma}{2}}\nabla\rho_n^{\frac{\gamma}{2}}(\nyb(\tilde{u}_n)-\nyb(u_n))\eta_n\psi \dd x \dd t.
    \end{aligned}
\end{equation}
Notice that the second and third line are uniformly bounded and have support contained in $\Qn$ due to the properties of $\nabla\eta_n$ and $\eta_n(\bdl(\tilde{u}_n)-\bdl(u_n))\neq 0$ only on $\Qn$. 
For the viscosity tensor we recover,
\begin{equation}\label{eq:approxME4}
    \begin{aligned}
        &\int_0^T2\nu\sqrt{\tilde{\rho}_n}\widetilde{\Svn}\nyb(\tilde{u}_n)\nabla\psi\dd x \dd t
        =\int_0^T\int_{\R^d}2\nu\sqrt{\rho_n}{\Svn}\nyb({u_n})\nabla(\psi\eta_n)\dd x \dd t\\
        &-\int_0^T\int_{\Qn}2\nu\sqrt{\tilde{\rho}_n}\Svn\nyb(\tilde{u}_n)\psi\nabla\eta_n\dd x \dd t\\ &+\int_0^T\int_{\Qn}2v\left(\sqrt{\tilde{\rho}_n}\nyb(\tilde{u}_n)-\sqrt{\rho_n}\nyb(u_n)\right)\Svn\nabla(\psi\eta_n) \dd x \dd t.
    \end{aligned}
\end{equation}
Similarly, for the capillary tensor, one has
\begin{equation}\label{eq:approxME5}
    \begin{aligned}
        &\int_0^T\int_{\R^d}2\kappa_n^2\sqrt{\tilde{\rho}_n}\widetilde{\Skn}\nyb(\tilde{u}_n)\nabla\psi\dd x \dd t
        =\int_0^T2\kappa_n^2\sqrt{\rho_n}{\Skn}\nyb({u_n})\nabla(\psi\eta_n)\dd x \dd t\\
        &-\int_0^T\int_{\Qn}2\kappa_n^2\sqrt{\tilde{\rho}_n}\Skn\nyb(\tilde{u}_n)\psi\nabla\eta_n \dd \dd t\\
        &+\int_0^T\int_{\Qn}2\kappa_n^2\left(\sqrt{\tilde{\rho}_n}\nyb(\tilde{u}_n)-\sqrt{\rho_n}\nyb(u_n)\right)\Skn\nabla(\psi\eta_n) \dd x \dd t.
    \end{aligned}
\end{equation}
Summing up equations from \eqref{eq:approxME0} to \eqref{eq:approxME5} yields that 
\begin{equation}\label{eq:approxME}
\begin{aligned}
    &\int_{\R^d}\tilde{\rho}_{n}^0\bdl(\tilde{u}_n^0)+\int_0^{T}\int_{\R^d}\tilde{\rho}_n\bdl(\tun) \partial_t\psi+\tilde{\rho}_n\tun\bdl(\tun)\cdot\nabla\psi-2\tilde{\rho}_n^{\frac{\gamma}{2}}\nabla\tilde{\rho}_n^{\frac{\gamma}{2}}\nyb(\tun)\psi \dd x \dd t\\ 
    &-\int_0^{T}\int_{\R^d}\left(2\nu{\sqrt{\tilde{\rho}_n}}\widetilde{\Svn}+2\kappa_n^2{\sqrt{\tilde{\rho}_n}}\widetilde{S_{\kappa_n,n}}\right)\nyb(\tun)\cdot\nabla\psi\dd x \dd t=\left\langle R_\beta^n,\eta_n\psi\right\rangle+\left\langle G_n,\psi\right\rangle,
    \end{aligned}
\end{equation}
where $R_\beta^n$ is the measure provided by Theorem \ref{thm:FEWStorus} and $G_n$ is a distribution such that 
\begin{equation}\label{eq:GN}
\begin{aligned}
    \left\langle G_n,\psi \right\rangle=&\int_{\R^d}\rho_n^0\left(\bdl(\tilde{u}_n^0)-\bdl(u_n^0)\right)\eta_n\psi(0,x)\dd x\\
    &+\int_0^T\int_{\R^d}\left(\bdl(\tilde{u}_n)-\bdl(u_n)\right)\left(\rho_n\partial_t(\eta_n\psi)+\rho_n u_n\cdot\nabla(\psi\eta_n)\right)-\rho_nu_n\bdl(\tilde{u}_n)\psi\nabla\eta_n\dd x\dd t\\
    &-\int_0^T\int_{\R^d}\left(2\nu\Svn+2\kappa_n^2\Skn\right)\sqrt{\tilde{\rho}_n}\nyb(\tilde{u}_n)\psi\nabla\eta_n \dd x \dd t\\
    &+2\int_0^T\int_{\R^d}\left(\sqrt{\tilde{\rho}_n}\nyb(\tilde{u}_n)-\sqrt{\rho_n}\nyb(u_n)\right)\left(\nu\Svn+\kappa_n\Skn\right)\nabla(\psi\eta_n)\dd x \dd t\\
   &-\int_0^T\int_{\R^d}\left(\tilde{\rho}_n^{\gamma}-\eta_n\rho_n^{\gamma}\right)\nabla(\nyb(\tilde{u}_n)\psi)\dd x \dd t\\
    &+\int_0^T\int_{\R^d}\rho_n^{\gamma}\nyb(\tilde{u}_n)\psi\nabla\eta_n+2\rho_n^{\frac{\gamma}{2}}\nabla\rho_n^{\frac{\gamma}{2}}(\bdl(\tilde{u}_n)-\bdl(u_n))\eta_n\psi \dd x \dd t.
\end{aligned}
\end{equation}
From the uniform bounds provided by Lemma \ref{lem:extensions}, the properties of $\beta_{\delta}^l$ and $\eta_n$ and the fact that $supp(\tilde{u}_n)\subset [-n,n]^d$, we conclude that  $supp(G_n)\subset \Qn$. In particular, there exists a uniform constant $C>0$ such that,
\begin{equation*}
    \left|\left\langle G_n,\psi\right\rangle\right|\leq C \|\psi\|_{C_c^{\infty}},
\end{equation*}
and arguing as for $D_n$, we observe that $G_n$ converges to $0$ in $\mathcal{D}'$ since for $n$ large enough $supp(G_n)\cap supp(\psi)=\emptyset$.
It remains to check that the compatibility conditions for the tensors $\widetilde\Svn$ and $\widetilde\Skn$ are satisfied. Let $\psi\in C_c^{\infty}([0,T)\times \R^d)$, then
\begin{align*}
       &\int_0^T\int_{\R^d}\sqrt{\nu}\sqrt{\tilde{\rho}_n}\nyb(\tilde{u}_n)_i[\widetilde\Tvn]_{jk}\psi\dd x\dd t=\int_0^T\int_{\R^d}\sqrt{\nu}\sqrt{\rho_n}\nyb(u_n)_i[\Tvn]_{jk}(\eta_n\psi)\dd x\dd t\\
       &+\int_0^T\int_{\R^d}\sqrt{\nu}\left(\sqrt{\tilde{\rho}_n}\nyb(\tilde{u}_n)_i-\sqrt{\rho_n}\nyb(u_n)_i\right)[\Tvn]_{jk}(\eta_n\psi)\dd x\dd t\\
       &=\int_0^T\int_{\R^d}\nu\left(\partial_j(\rho_n\nyb(u_n)_i u_{n,k})-2\sqrt{\rho_n}u_n\nyb(u_n)_i \partial_j\sqrt{\rho_n} \right)(\eta_n\psi)\dd x\dd t+\left\langle \overline{R_{\beta}}, \eta_n\psi\right\rangle \\
       &+\int_0^T\int_{\R^d}\sqrt{\nu}\left(\sqrt{\tilde{\rho}_n}\nyb(\tilde{u}_n)_i-\sqrt{\rho_n}\nyb(u_n)_i\right)[\Tvn]_{jk}(\eta_n\psi)\dd x\dd t\\
       &=\int_0^T\int_{\R^d}\nu\left(\partial_j(\tilde{\rho}_n\nyb(\tilde{u}_n)_i\tilde{u}_{n,k})-2\sqrt{\tilde{\rho}_n}u_n\nyb(\tilde{u}_n)_i\partial_j\sqrt{\tilde{\rho}_n}\right)\psi\dd x \dd t+\left\langle\overline{R_{\beta}},\eta_n\psi\right\rangle\\
       &+\int_0^T\int_{\Qn}\nu\left(\eta_n\partial_j(\rho_n\nyb(u_n)_i u_{n,k})-\partial_j(\tilde{\rho}_n\nyb(\tilde{u}_n)_i \tilde{u}_{n,k})\right)\dd x\dd t\\
       &-\int_0^T\int_{\Qn} 2\nu\left(\eta_n\left(\sqrt{\rho_n}u_n\nyb(u_n)_i\partial_j\sqrt{\rho_n}\right)+\sqrt{\tilde{\rho}_n}u_n\nyb(\tilde{u}_n)_i\partial_j\sqrt{\tilde{\rho}_n}\right)\psi\dd x\dd t\\
       &+\int_0^T\int_{\Qn}\sqrt{\nu}\left(\sqrt{\tilde{\rho}_n}\nyb(\tilde{u}_n)_i-\sqrt{\rho_n}\nyb(u_n)_i\right)[\Tvn]_{jk}(\eta_n\psi)\dd x\dd t.
    \end{align*}
Thus, there exists a distribution $V_n$ with $supp(V_n)\subset \Qn$ and such that for any $\psi\in C_c^{\infty}([0,T)\times\R^d)$ one has
\begin{equation*}
\begin{aligned}
    &\int_{0}^T\int_{\R^d}\sqrt{\nu}\sqrt{\tilde{\rho}_n}\nyb(\tilde{u}_n)_i[\widetilde\Tvn]_{jk}\psi\dd x\dd t\\
    &=\nu\int_0^T\int_{\R^d}\left(\partial_j(\tilde{\rho}_n\nyb(\tilde{u}_n)_i\tilde{u}_{n,k})-2\sqrt{\tilde{\rho}_n}u_n\nyb(\tilde{u}_n)_i\partial_j\sqrt{\tilde{\rho}_n}\right)\psi\dd x \dd t\\
    &+\left\langle \overline{R_{\beta}},\eta_n\psi\right\rangle+\left\langle V_n,\psi\right\rangle.
    \end{aligned}
\end{equation*}
Moreover, there exists a uniform $C>0$ such that
\begin{equation*}
    \left|\left\langle V_n,\psi\right\rangle\right|\leq C\|\psi\|_{C_c^{\infty}},
\end{equation*}
for any $\psi\in C_c^{\infty}([0,T)\times\R^d)$. The uniform bound in $\mathcal{D}'$ together with the support properties imply that $V_n$ converges to $0$ as $n\rightarrow\infty$.
Arguing similarly, we recover for the capillary tensor
\begin{equation*}
\begin{aligned}
    &\int_0^T\int_{\R^d}2\kappa_n\sqrt{\tilde{\rho}_n\widetilde{\Skn}}\psi\dd x\dd t=\kappa^2\int_{0}^T\int_{\R^d}\sqrt{\rho_n}\left(\nabla^2\sqrt{\rho_n}-4(\nabla\rho_n^{\frac{1}{4}}\otimes\nabla\rho_n^{\frac{1}{4}})\right)\eta_n\psi\dd x\dd t\\
    &+2\kappa_n\int_0^T\int_{\R^d}(\sqrt{\tilde\rho_n}-\sqrt{\rho_n})\Skn\eta_n\psi\dd x\dd t\\
    &=\kappa_n^2\int_{0}^T\int_{\R^d}\sqrt{\tilde{\rho}_n}\left(\nabla^2\sqrt{\tilde{\rho}_n}-4(\nabla\tilde{\rho}_n^{\frac{1}{4}}\otimes\nabla\tilde{\rho}_n^{\frac{1}{4}})\right)\psi\dd x\dd t\\
    &+\kappa_n^2\int_{0}^T\int_{\Qn}\left(\sqrt{\rho_n}\left(\nabla^2\sqrt{\rho_n}-4(\nabla\rho_n^{\frac{1}{4}}\otimes\nabla\rho_n^{\frac{1}{4}})\right)\eta_n-\sqrt{\tilde{\rho}_n}\left(\nabla^2\sqrt{\tilde{\rho}_n}
    -4(\nabla\tilde{\rho}_n^{\frac{1}{4}}\otimes\nabla\tilde{\rho}_n^{\frac{1}{4}})\right)\right)\psi\dd x\dd t\\
    &+2\kappa_n\int_0^T\int_{\Qn}(\sqrt{\tilde{\rho}_n}-\sqrt{\rho_n})\Skn\eta_n\psi\dd x\dd t
    \end{aligned}
\end{equation*}
Hence, the error is given by a distribution $K_n$ such that $supp(K_n)\subset \Qn$ and $K_n$ is uniformly bounded in $\mathcal{D}'$ and converges to $0$ as $n$ goes to infinity. We conclude that $(\tilde{\rho}_n,\tilde{u}_n)$ as defined in \eqref{eq:extension} is an approximate truncated weak solution to \eqref{eq:QNSFEWS} with initial data \eqref{eq:datatilde}. \\
It remains to perform the limit as $n$ goes to infinity. Since \eqref{eq:energyextension} and \eqref{eq:BDextensions} are verified, the inequalities \eqref{eq:energyQNS} and \eqref{ineq:BDEQNS} follow from the weak convergences provided by \eqref{eq:extensionwc}, \eqref{eq:uniformextensionk}, the second statement of Lemma \ref{lem:convergenceQNSren}, lower semi-continuity of norms and the following observation. Denote by $\Lambda$ the weak-$\ast$ $L^{\infty}L^2$-limit of $\sqrt{\tilde{\rho}_n}\tilde{u}_n$. Then, since $\Lambda=\sqrt{\rho}u$ whenever $\rho\neq 0$, we infer for a.e. $t\in [0,T)$,
\begin{equation*}
    \int_{\R^d}\frac{1}{2}\rho|u|^2(t)\dd x\leq  \int_{\R^d}\frac{1}{2}|\Lambda|^2(t)\dd x.
\end{equation*}
Next, we wish to pass to the limit in \eqref{eq:approxCE}. We check that 
\begin{equation*}
\begin{aligned}
    \int_{\R^d}\tilde{\rho}_n^0\psi(0,x)\dd x
    &=\int_{\R^d}\rho_n^0\eta_n\psi(0,x)\dd x+\int_{\R^d}(1-\eta_n)\psi(0,x)\dd x\\
    &\rightarrow \int_{\R^d}\rho^0\psi(0,x)\dd x,
    \end{aligned}
\end{equation*}
where the first term converges thank to Lemma \ref{lem:data} and the second term converges to $0$ since for any $\psi$ there exists $n_0\in \N$ such that for all $n\geq n_0$, one has $supp(\psi)\cap supp((1-\eta_n))=\emptyset$. Lemma \ref{lem:convergenceQNS} allows thus to 
pass to the limit on the left-hand-side of \ref{eq:approxCE}. Hence, $(\rho,u)$ satisfies the continuity equation of \eqref{eq:QNSFEWS}. We proceed to the limit of \eqref{eq:MEtrunc}. Firstly, we infer that 
\begin{equation*}
\int_{\R^d}\tilde{\rho}_n^0\bdl(\tilde{u}_n^0)\psi(0,x)\dd x\rightarrow\int_{\R^d}\rho^0\bdl(u_n^0)\psi(0,x)\dd x,
\end{equation*}
by splitting the integral as in \eqref{eq:approxME0}, applying Lemma \ref{lem:data} to the first term and arguing with the support properties to dispose of the remainder. Indeed, $(\bdl(\tilde{u}_n)-\bdl(u_n))\eta_n\neq 0$ only on $\Qn$ and thus
\begin{equation*}
    \int_{\R^d}\rho_n^0(\bdl(\tilde{u}_n)-\bdl(u_n))\eta_n\psi(0,x)\dd x\rightarrow 0.
\end{equation*}
The last term converges to $0$ as $\bdl \in L^{\infty}$ and for sufficiently large $n$ $supp(\psi)\cap supp((1-\eta_n))=\emptyset$. By applying the convergence results of Lemma \ref{lem:convergenceQNS} and Lemma \ref{lem:convergenceQNSren}, we may pass to the limit on the left-hand-side of \eqref{eq:approxME}. We refer the reader to \cite{LV16} and also \cite{AS19} for more details. 
Since $R_\beta^n\eta_n$ is uniformly bounded, there exists a measure $\mu_\beta$ such that 
\[
\left\langle R_\beta^n\eta_n,\psi\right\rangle\rightarrow\left\langle \mu_\beta,\psi\right\rangle. \]
Thus, $(\rho,u)$ satisfies \eqref{eq:approxCE}, \eqref{eq:MEtrunc} with $F=G=0$. Similarly, since $\overline{R_\beta}^n\eta_n$ is uniformly bounded, there exists a measure $\overline{\mu_\beta}$ such that 
\[
\left\langle \overline{R_\beta}^n\eta_n,\psi\right\rangle\rightarrow\left\langle \overline{\mu_\beta},\psi\right\rangle.
\]
Given that $V_n, K_n$ converge to $0$ in the limit as $n\rightarrow \infty$, we infer that the compatibility conditions are satisfied in the limit.   
\end{proof}

\section{Proof of the main Theorems}\label{sec:weak}
It remains to show that the truncated finite energy weak solution $(\rho,u)$ provided by Theorem \ref{thm:extensions} is a finite energy weak solution to \eqref{eq:QNSFEWS} with \eqref{eq:farfield}. We proceed as in \cite{LV16}. 
The proof uses a local argument for which the far-field plays a minor role, we omit the details.

\begin{proof}[Proof of Theorem \ref{thm:mainQNS}]
Given initial data $(\sqrt{\rho^0},\sqrt{\rho^0}u^0)$ of finite energy, i.e. such that \eqref{eq:energy} is bounded, concatenating Lemma \ref{lem:data}, Theorem \ref{thm:FEWStorus} and Theorem \ref{thm:extensions} provides a finite energy truncated weak solution $(\rho,u)$ to \eqref{eq:QNSFEWS} with far-field condition \eqref{eq:farfield} in the sense of Definition \ref{defi:truncatedsolution}. It remains to show that $(\rho,u)$ is a finite energy weak solution to \eqref{eq:QNSFEWS} according to Definition \ref{defi:FEWS}. For that purpose, we recall that $\beta_{\delta}^l$ in \eqref{eq:MEtrunc} is as stated in Definition \ref{defi:truncation}. From Lemma \ref{lem:truncation} and the uniform bounds , we conclude by applying the dominated convergence Theorem that for any $1\leq l\leq 3$ and any compact set $K\subset \R^d$,
\begin{equation*}
\begin{aligned}
   & \rho\beta_\delta^l(u)\rightarrow\rho u \qquad \text{in} \quad L^1((0,T)\times K), \qquad \sqrt{\rho}\beta_\delta^l(u)\rightarrow\sqrt{\rho} u^l \qquad \text{in} \quad L^{2}((0,T)\times K),\\
   & \rho^{\frac{\gamma}{2}}\nabla\rho^{\frac{\gamma}{2}}\nabla_y\beta_{\delta}^l(u)\rightarrow \rho^{\frac{\gamma}{2}}\nabla\rho^{\frac{\gamma}{2}} \qquad \text{in} \quad L^1((0,T)\times K),\\
  &  \sqrt{\rho}(\sqrt{\nu}\Sv+\kappa\Sk)\nabla_y\beta_\delta^l(u)\rightarrow \sqrt{\rho}(\sqrt{\nu}\Sv+\kappa\Sk)\qquad \text{in} \quad L^1((0,T)\times K).
\end{aligned}
\end{equation*}
Further we have that 
\begin{equation*}
    \|\mu_{\beta}\|_{\mathcal{M}}+ \|\overline{\mu}_{\beta}\|_{\mathcal{M}}\leq 2M\delta.
\end{equation*}
Thus, performing the $\delta$-limit in \eqref{eq:MEtrunc} yields a weak solution to the momentum equation of \eqref{eq:QNSFEWS} which completes the proof.
\end{proof}

Next, we comment on the proof of Corollary \ref{coro:degenerate}.
\begin{proof}[Proof of Corollary \ref{coro:degenerate}]
If $\kappa=0$, the BD entropy of the initial data yields a $L^2(\R^d)$ bound for $\nabla\sqrt{\rho^0}$. We may then construct periodic initia data on $\T_n^d$ by means of Lemma \ref{lem:data}. Defining $\kappa_n$ as in \eqref{eq:kappan}, concatenating Theorem \ref{thm:FEWStorus} and  Theorem \ref{thm:extensions} provides a truncated finite energy weak solution to \eqref{eq:QNSFEWS}. Furthermore, the capillary tensor $\widetilde\Skn$ is uniformly bounded in $L_{loc}^2(0,T;L^2(\R^d))$ and since $\kappa_n\rightarrow 0$ the corresponding contribution in \eqref{eq:approxME} satisfies
\begin{equation*}
    \int_{0}^T\int_{\R^d}2\kappa_n\widetilde{\Skn}\nyb(\tilde{u}_n)\cdot \nabla\psi \dd x\dd t \rightarrow 0.
\end{equation*}
Thus, the pair $(\sqrt{\rho},u)$ is a truncated weak solution to \eqref{eq:QNSFEWS} with $\kappa=0$. Proceeding as in the proof of Theorem \ref{thm:mainQNS} we carry out the $\delta$-limit to obtain a finite energy weak solution.
\end{proof}

Finally, we sketch the proof of Theorem \ref{thm:mainQNS2}. It follows the strategy of proof of Theorem \ref{thm:mainQNS}. Minor modifications are necessary to adapt the method to the different far-field behavior. However, gaining integrability for $\rho$ in the present setting simplifies the proof.

\begin{proof}[Proof of Theorem \ref{thm:mainQNS2}]
If the system is considered with trivial far-field behavior, then the internal energy is given by \eqref{eq:energygammalaw}. Let $(\sqrt{\rho^0},\sqrt{\rho^0}u^0)$ be initial data of finite energy and BD-entropy. For $\eta_n$ as in \eqref{eq:eta}, let
\begin{equation*}
    \sqrt{\rho_n^0}= \sqrt{\rho_n^0}\eta_n, \qquad  \sqrt{\rho_n^0}u_n^0= \sqrt{\rho_n^0}u_n^0\eta_n.
\end{equation*}
One easily verifies that the respective version of Lemma \ref{lem:data} is still valid and provides periodic initial data. We then exploit Theorem \ref{thm:FEWStorus} to prove the respective version of Theorem \ref{thm:extensions} yielding a truncated finite energy weak solution on the whole space.  The renormalized internal energy is substituted by \eqref{eq:energygammalaw} and we adapt the related uniform bounds. Notice that the extensions \eqref{eq:extension} are defined by the stationary solution $(\rho=0, u=0)$ away from $[-n,n]^d$.
The $\delta$-limit is then performed analogously to the proof of Theorem \ref{thm:mainQNS}.
\end{proof}

\section*{Acknowledgements} The authors acknowledge support by INdAM-GNAMPA through the project "Esistenza, limiti singolari e comportamento asintotico per equazioni Eulero/Navier--Stokes--Korteweg". This paper was mostly prepared while the second author was a PhD student at GSSI, which he acknowledges. The second author would also like to acknowledge the University of L'Aquila for the support through the scholarship "Studio di strutture coerenti in fluidi quantistici". We thank Christophe Lacave for the careful reading of the paper and several useful comments. 



\begin{thebibliography}{99}
\bibitem{AHM18} P. Antonelli, L.E. Hientzsch, P. Marcati, \emph{On the low Mach number limit for Quantum Navier-Stokes equations}, preprint \href{https://arxiv.org/pdf/1902.00402.pdf}{arXiv:1902.00402} (2019).
\bibitem{AHMHyp} P. Antonelli, L.E. Hientzsch, P. Marcati, \emph{The incompressible limit for finite energy weak solutions of quantum Navier–Stokes equations}, to appear in Proceedings of the XVII International Conference on Hyperbolic Problems: Theory, Numerics, Applications 2018.
\bibitem{AHM} P. Antonelli, L.E. Hientzsch, P. Marcati, \emph{On the Cauchy problem for the QHD system with infinite mass and energy: applications to quantum vortex dynamics}, in preparation.
\bibitem{AM}  P. Antonelli, P. Marcati, \emph{On the finite energy weak solutions to a system in Quantum Fluid Dynamics}, Comm. Math. Phys. {\bf 287} (2009), no. 2, 657--686.
\bibitem{AM16} P. Antonelli, P. Marcati, \emph{Some results on systems for quantum fluids}, (2016), Recent Advances in Partial Differential Equations and Application, Contemp. Math. {\bf666}, 41--54.
\bibitem{AS15} P. Antonelli, S. Spirito, \emph{On the compactness of finite energy weak solutions to the quantum Navier-Stokes equations}, J. Hyp. Diff. Equ. {\bf15}, no. 1 (2018), 133--147. 
\bibitem{AS} P. Antonelli, S. Spirito, \emph{Global existence of finite energy weak solutions of quantum Navier-Stokes equations}, Arch. Rat. Mech. Anal. {\bf 225}, no.3 (2017), 1161--1199.
\bibitem{AS18} P. Antonelli, S. Spirito, \emph{On the compactness of weak solutions to the Navier-Stokes-Korteweg equations for capillary fluids}, Nonlinear Anal. {\bf 187} (2019), 110--124. 
\bibitem{AS19} P. Antonelli, S. Spirito, \emph{Global existence of weak solutions to the Navier--Stokes--Korteweg equations}, preprint \href{https://arxiv.org/abs/1903.02441}{arXiv:1903.02441} (2019).
\bibitem{BD} D. Bresch, B. Desjardins, \emph{Quelques mod{\`e}les diffusifs capillaires de type {K}orteweg}, C. R. M{\'e}canique {\bf 332}, no. 11 (2004), 881--886.
\bibitem{BDL} D. Bresch, B. Desjardins, C.-K. Lin, \emph{On some compressible fluid models: {K}orteweg, lubrication, and shallow water systems}, Comm. Partial Differential Equations {\bf 28}, no. 3-4 (2003), 843--868.
\bibitem{BJ} D. Bresch, P.-E. Jabin, \emph{Global existence of weak solutions for compressible Navier-Stokes equations: thermodynamically unstable pressure and anisotropic viscous stress tensor}, Ann. of Math. {\bf 188}, no. 2 (2018), 577--684. 
\bibitem{BNV16} D. Bresch, P. Noble, J.-P. Vila, \emph{Relative entropy for compressible {N}avier-{S}tokes equations with density-dependent viscosities and applications}, C. R. Math. Acad. Sci. Paris {\bf 354}, no. 1 (2016), 45--49.
\bibitem{BVY19} D. Bresch, A. Vasseur, C. Yu, \emph{Global Existence of Entropy-Weak Solutions to the Compressible Navier-Stokes Equations with Non-Linear Density Dependent Viscosities}, preprint \href{https://arxiv.org/abs/1905.02701}{arXiv:1905.02701} (2019).
\bibitem{BM} S. Brull, F. M{\'e}hats, \emph{Derivation of viscous correction terms for the isothermal quantum Euler model}, Z. Angew. Math. Mech. {\bf 90} (2010), 219--230. 
\bibitem{CCH} R. Carles, K. Carrapatoso, M. Hillairet, \emph{Global weak solutions for quantum isothermal fluids}, preprint \href{https://arxiv.org/abs/1905.00732}{arXiv:1905.0073} (2019).
\bibitem{DS} J.E. Dunn, J. Serrin, \emph{On the thermomechanics of interstitial working},  Arch. Rational Mech. Anal. {\bf 88}, no.2 (1985), {95--133}.
\bibitem{FN} E. Feireisl, A. Novotn{\`y}, \emph{Singular limits in thermodynamics of viscous fluids}, Springer, 2009. \bibitem{FN01} E. Feireisl, A. Novotn\`{y}, H.  Petzeltov\'{a}, \emph{On the existence of globally defined weak solutions to the {N}avier-{S}tokes equations}, J. Math. Fluid Mech. {\bf 3}, no. 4 (2001), 358--392.
\bibitem{GS18} Z. Guo, W. Song, \emph{Global well-posedness and large-time behavior of classical solutions to the 3D Navier-Stokes system with changed viscosities}, J. Math. Phys. {\bf 60} (2019). 
\bibitem{H18} B. Haspot, \emph{Existence of global strong solution for the compressible {N}avier-{S}tokes equations with degenerate viscosity coefficients in 1{D}}, Math. Nachr. {\bf 291}, no. 14--15 (2018), 2188--2203.
\bibitem{J} A. J\"ungel, \emph{Dissipative quantum fluid models}, Riv. Mat. Univ. Parma {\bf 3} (2012), 217--290.
\bibitem{JMi} A. J\"ungel, J.-P. Mili{\v{s}}ic, \emph{Full compressible {N}avier-{S}tokes equations for quantum fluids: derivation and numerical solution}, Kin. Relat. Mod. {\bf 4}, no.3 (2011), 785--807. 
\bibitem{K} D.J. Korteweg, \emph{Sur la forme que prennent les {\'e}quations du mouvements des fluides si l'on tient compte des forces capillaires caus{\'e}es par des variations de densit{\'e} consid{\'e}rables mais continues et sur la th{\'e}orie de la capillarit{\'e} dans l'hypoth{\`e}se d'une variation continue de la densit{\'e}}, Arch. Neerl. Sci. Exactes {\bf 6} (1901), 1--24.
\bibitem{LV16} I. Lacroix-Violet, A. Vasseur, \emph{Global weak solutions to the compressible quantum Navier-Stokes and its semi-classical limit}, J. Math. Pures Appl. {\bf 114} (2018), 191--210.
\bibitem{LL} L. Landau, E. Lifschitz \emph{Quantum Mechanics: Non-relativistic Theory}. Pergamon Press, New York, 1977.
\bibitem{L} J. Leray, \emph{Etude de diverses {\'e}quations int{\'e}grales non lin{\'e}aires et de quelques probl{\`e}mes que pose l'hydrodynamique}, J. Math. Pures Appl. {\bf 12} (1933), 1--82.
\bibitem{LPZ} Y. Li, R. Pan, S. Zhu, \emph{On Classical Solutions for Viscous Polytropic Fluids with Degenerate Viscosities and Vacuum}, Arch. Rat. Mech. Anal. (2019).
\bibitem{LX} J. Li, Z. Xin, \emph{Global existence of weak solutions to the barotropic compressible Navier-Stokes flows with degenerate viscosities}, preprint \href{www.arxiv.org/abs/1504.06826}{arXiv:1504.06826} (2015).
\bibitem{LPZ16} Y. Li, R. Pan, S. Zhu, \emph{Recent progress on classical solutions for compressible isentropic {N}avier-{S}tokes equations with degenerate viscosities and vacuum}, Bull. Braz. Math. Soc. (N.S.) {\bf 47}, no. 2 (2016), 507--519.
\bibitem{L96} P.-L. Lions, \emph{Mathematical topics in fluid mechanics. {V}ol. 2, Compressible Models}, Oxford Lecture Series in Mathematics and its Applications {\bf 3}, The Clarendon Press, Oxford University Press, New York (1996).
\bibitem{LM98} P.-L. Lions, N. Masmoudi \emph{Incompressible limit for a viscous compressible fluid}, J. Math. Pures Appl. {\bf 77}, no. 6 (1998), 585 -- 627.
\bibitem{LZZ} B. L{\"u}, R. Zhang, X. Zhong, \emph{Global existence of weak solutions to the compressible quantum Navier-Stokes equations with degenerate viscosity}, Journal of Mathematical Physics {\bf 60} (2019).
\bibitem{MV08} A. Mellet, A. Vasseur, \emph{Existence and uniqueness of global strong solutions for one-dimensional compressible {N}avier-{S}tokes equations}, SIAM J. Math. Anal. {\bf 39}, no. 4 (2007/08), 1344--1365.
\bibitem{VY16} A. Vasseur, C. Yu, \emph{Global Weak Solutions to the Compressible Quantum Navier--Stokes Equations with Damping}, SIAM J. Math. Anal. {\bf 48}, no. 2 (2016), 1489--1511. 
\end{thebibliography}

\bibliographystyle{plain}

\end{document}